\documentclass[a4paper]{article}
\usepackage{amsthm,amssymb,amsmath,enumerate}
\usepackage{algorithm}
\usepackage{thmtools, thm-restate}
\usepackage{url}
\usepackage{tikz}
\usetikzlibrary{backgrounds}
\usetikzlibrary{decorations}
\usetikzlibrary{decorations.pathmorphing}
\usepackage{ifthen}

\tikzstyle{hvertex}=[thick,circle,inner sep=0.cm, minimum size=2mm, fill=white, draw=black]
\tikzstyle{hedge}=[very thick]
\tikzstyle{rededge}=[very thick,red]
\tikzstyle{point}=[draw,circle,inner sep=0.cm, minimum size=1mm, fill=black]
\tikzstyle{pointer}=[thick,->,shorten >=2pt,color=hellgrau]
\tikzstyle{pathedge}=[hedge,decorate, decoration={random steps,segment length=3pt,amplitude=1pt}]

\colorlet{auchblau}{blue!60!white}
\colorlet{hellblau}{blue!20!white}
\colorlet{hellrot}{red!40!white}
\colorlet{hellgrau}{black!30!white}
\colorlet{dunkelgrau}{black!60!white}

\colorlet{grau}{black!50!white}

\newtheorem{definition}{Definition}
\newtheorem{proposition}[definition]{Proposition}
\newtheorem{theorem}[definition]{Theorem}
\newtheorem{corollary}[definition]{Corollary}
\newtheorem{lemma}[definition]{Lemma}

\newcommand{\cP}{\mathcal{P}}
\newcommand{\cQ}{\mathcal{Q}}

\newcommand{\cT}{\mathcal{T}}

\newcommand{\cR}{\mathcal{R}}

\newcommand{\cH}{\mathcal{H}}

\newcommand{\comment}[1]{}
\newcommand{\N}{\mathbb N}

\newcommand{\Z}{\mathbb Z}

\newcommand{\m}{\textrm{mod}\ m}

\newcommand{\EPP}{Erd\H{o}s-P\'osa property}
\newcommand{\eptan}{\ensuremath{\mathcal T_{\text{\rm EP}}}}

\title{Zero $A$-paths and the \EPP}
\author{Arthur Ulmer\thanks{supported by DFG, grant no.\  321904558}}
\date{}

\begin{document}
\maketitle

\begin{abstract}
Let $\Gamma$ be an Abelian group.
In this paper I characterize the $A$-paths of weight $0\in\Gamma$
that have the Erd\H{o}s-P\'osa property. 
Using this in an auxiliary graph, one can also easily characterize the $A$-paths of weight $\gamma\in\Gamma$ that have the \EPP.
These results also extend to long paths, that is paths of some minimum length.

A structural result on zero walls with non-zero linkages will be proven as a means to prove the main result of this paper. This immediately implies that zero cycles with respect to an Abelian group $\Gamma$ have the \EPP.
\end{abstract}

\section{Introduction}
Menger's famous theorem describes a relation between the maximum number of paths $p_{max}$ between two sets $A$ and $B$ and the minimum number $s_{min}$ of vertices that separates these two sets. It is clear that $s_{min}\geq p_{max}$ since one vertex of each disjoint path has to be in any separator. The theorem of Menger states that we can also bound $s_{min}$ from above by $p_{max}$.
From this we obtain that for any positive integer $k$ and any graph $G$ with vertex sets $A$ and $B$, there are either $k$ disjoint $A$-$B$-paths in $G$ or a set of at most $k-1$ vertices that intersect all $A$-$B$-paths.
We say that $A$-$B$-paths have the \EPP. More generally, we say that a class of graphs $\cH$ has the \emph{(vertex-)\EPP}\ if there is a function $f:\mathbb{N}\to\N$ such that in every graph $G$ and for every $k\in\N$, there are either $k$ (vertex-)disjoint subgraphs of $G$ that belong to $\cH$ or a set of at most $f(k)$ vertices that intersects each such subgraph in $G$. We call $f$ a \emph{hitting set function} for $\cH$. Note that this definition technically does not encompass $A$-$B$-paths; we extend the definition in the obvious way. 

Gallai \cite{Gal61} has shown that there is such a relation for $A$-paths, too (an $A$-path is a path that connects two different vertices of a vertex set $A$).
For any integer $k$ and any graph $G$ with a vertex set $A\subseteq V(G)$, there are either $k$ disjoint $A$-paths in $G$ or a set of at most $2k-2$ vertices that intersects all $A$-paths.
Let $\Gamma$ be some Abelian group and label each edge of a graph $G$ with an element of $\Gamma$. A non-zero $A$-path (with respect to $\Gamma$) is an $A$-path such that the sum of the weights of its edges is non-zero.
Wollan \cite{Wol10} has shown that for all groups $\Gamma$, the non-zero $A$-paths (with respect to $\Gamma$) have the \EPP. So what about zero $A$-paths? For some special cases, this has already been looked at.
For fixed integers $m$, Bruhn et al. proved in \cite{BHJ18b} that $A$-paths of length $0\ \m$ do not have the \EPP\ if $m$ is non-prime and $m\neq 4$. They also showed that even $A$-paths have the \EPP, that is $A$-paths of length $0\ \textrm{mod}\ 2$. Furthermore, Bruhn and U. \cite{BU18} showed that $A$-paths of length $0\ \textrm{mod}\ 4$ have the \EPP. In this paper I will show:

\begin{restatable}{theorem}{mainThm}\label{mainThm}
Let $\Gamma$ be an Abelian group. The zero $A$-paths (with respect to $\Gamma$) have the \EPP\ if and only if:
\begin{itemize}
\item $\Gamma$ is finite and
\item for all $x,y\in\Gamma$ such that $y\neq 0$, there is an $n\in\Z$ such that $2x+ny=0$.
\end{itemize}
\end{restatable}

This answers a question by Bruhn et al \cite{BHJ18b}. They wanted to know for odd primes $m$ whether the $A$-paths of length $0\ \m$ have the \EPP.

\begin{restatable}{corollary}{primeCor}\label{primeCor}
Let $m$ be a fixed odd prime. The $A$-paths of length $0\ \m$ have the \EPP.
\end{restatable}

From Theorem \ref{mainThm}, one can quite easily deduce the following corollary:
\begin{restatable}{corollary}{mainCor}\label{mainCor}
Let $\Gamma$ be an Abelian group and let $\gamma\in\Gamma$. The $A$-paths of weight $\gamma$ have the \EPP\ if and only if:
\begin{itemize}
\item $\Gamma$ is finite and
\item for all $x,y\in\Gamma$ such that $y\neq 0$, there is an $n\in\Z$ such that $2x+ny=\gamma$.
\end{itemize}
\end{restatable}

The idea of the proof of Theorem \ref{mainThm} is to assume that it is false, which implies that there is some counterexample. In \cite{BU18}, it has been shown that there is then a large tangle in that counterexample and from a theorem of Robertson and Seymour, it follows that there is a large wall. Then the following result, which will be proven in this paper, can be applied, which yields a nice structure. If $\Gamma$ is a fixed Abelian group, whenever something is said to be zero, it means that it is zero with respect to $\Gamma$. We will explain everything else in detail later.

\begin{restatable}{theorem}{primeWall}\label{primeWall}
Let $\Gamma$ be a finite Abelian group such that $m=|\Gamma|$ and let $r,s$ be some positive integers.
There are integers $f_{\ref{primeWall}}(m,r,s)$ and $g_{\ref{primeWall}}(s)$ such that if a graph $G$ contains a wall $W$ of size at least $f_{\ref{primeWall}}(m,r,s)$, then there is a zero wall $W_0$ of size at least $r$ in $G$ and either 
\begin{itemize}
\item $W_0$ has a pure non-zero linkage $\mathcal{P}$ of size $s$ or
\item there is a vertex set $X$ of size at most $g_{\ref{primeWall}}(s)$ that is disjoint from $W_0$ such that all paths in $G-X$ between the branch vertices of $W_0$ have weight zero.
\end{itemize}
Additionally, the tangle $\cT_{W_0}$ is a truncation of $\cT_W$.
\end{restatable}

Basically, if there is a large wall in a graph $G$ then there is a wall $W_0$ where all paths in $W_0$ between its branch vertices are zero and either there are some non-zero paths that are attached nicely to $W_0$ or there is a vertex set $X$ such that, not only all paths in $W_0$, but all paths in $G-X$ between the branch vertices of $W_0$ are zero. 

This result implies the following:
\begin{theorem}
Let $\Gamma$ be an Abelian group. The zero cycles (with respect to $\Gamma$) have the \EPP.
\end{theorem}

Again this is in contrast to a result by Wollan on non-zero cycles.
He proved that the non-zero cycles have the \EPP\ if and only if $\Gamma$ is an Abelian group that does not contain any elements of order $2$ \cite{Wol11}.

As an aside, there is an edge-version of the \EPP; just replace all occurences of "vertex/vertices" by "edge/edges". 
It is quite a bit harder and this property seems to break down more easily.
It is still true that $A$-paths have the edge-\EPP \cite{Mad78}.
However, neither $A$-paths of any fixed weight $\gamma$ nor non-zero $A$-paths have the \EPP\ for any choice of $\Gamma$.
For a large summary of classes of graph that have the vertex-/edge-\EPP\ (or do not) see \cite{BHJ18b} and \cite{RaymondWeb}.

\section{Premliminaries}
In this paper I will use the notation of Diestel \cite{diestelBook17}.
For the rest of the paper let $\Gamma$ be a fixed Abelian group.
Let $G$ be a graph.
A mapping $\gamma:\Gamma\to E(G)$ is a \emph{$\Gamma$-labelling} in $G$. If there is a $\Gamma$-labelling in a graph, we say that this graph is $\Gamma$-labelled. Whenever a graph is mentioned, it is assumed to be $\Gamma$-labelled.
The \emph{weight} of an edge $e\in E(G)$ is $\gamma (e)$.

Let $A\subseteq V(G)$, we write $G[A]$ for the induced subgraph of $G$ on $A$. We denote by $G-A$ the subgraph $G[V(G)\setminus A]$.
Furthermore, if $H$ is a graph,
we write $G-H$ for $G-V(H)$.

\subsection{Groups}
The groups in this paper are additive
with neutral element $0$.
For a positive integer $n$, we define $ny:=\sum_{i=1}^n y$
For a negative integer, we define $ny:=|n|(-y)$.
Lastly, we define $0\cdot y:=0$.
The \emph{order} of an element $x\in\Gamma$ is the smallest positive integer such that $nx=0$. If there is no such $n$, we say that the order of $x$ is \emph{infinite}. Note that this may only happen if $\Gamma$ is infinite.
We denote by $\Gamma_2$ the subgroup of $\Gamma$ that contains all elements of $\Gamma$ of order at most $2$. Hence, $\Gamma_2$ contains exactly the elements $\gamma\in\Gamma$ such that $2\gamma=0$.

We denote by $\Z_m$ the group with elements $\{0,\ldots,m-1\}$ and with binary operation $x+y=x+y\ \m$. 

\subsection{Paths}
The \emph{weight} of a path is the sum of the weights of its edges.
A path is \emph{trivial} if it is edgeless, which means it contains exactly one vertex, otherwise it is \emph{non-trivial}. We define the weight of a trivial path to be $0\in\Gamma$. A path is \emph{zero} (with respect to $\Gamma$) if its weight is $0\in\Gamma$ and \emph{non-zero} (with respect to $\Gamma$) otherwise. 
For $\gamma\in\Gamma$, we define a \emph{$\gamma$-path} as a path of weight $\gamma$.

For a path $P$ that contains two vertices $a$ and $b$ we define $aPb$ as the subpath of $P$ between $a$ and $b$. 
If additionally another path $Q$ contains $b$ and one more vertex $c$, then $aPbQc$ is the graph that we obtain by merging $aPb$ and $bQc$. Whenever we use this notation, the resulting graph will be a path.

The \emph{interior} of a path are all vertices except its endvertices.
Let $A$ be a set of vertices in a graph. We define an \emph{$A$-path} as a non-trivial path with both endvertices in $A$ that does not contain any vertices of $A$ in its interior. An \emph{$A$-$B$-path or a path from $A$ to $B$} is a path with one endvertex in $A$ and the other in $B$, while its interior is disjoint from both $A$ and $B$.

The following theorem by Wollan on non-zero $A$-paths was already mentioned in the introduction:

\begin{theorem}[Wollan \cite{Wol10}]\label{Wollan}
The non-zero $A$-paths have the \EPP\ with hitting set function $f_{\ref{Wollan}}(k)=50k^4$.
\end{theorem}

Let $A$ and $B$ be two vertex sets in a graph $G$. An \emph{$A$-$A$-$B$-path} is a path with one endvertex in $A$, the other in $B$ and one vertex of $A$ in its interior. It may contain more vertices of $A$ or $B$ in its interior. The characterization of Bruhn et al. in \cite{BJU20}, implies the following proposition.
\begin{proposition}[Bruhn et al. \cite{BJU20}]\label{AABpath}
$A-A-B$-paths have the \EPP\ with hitting set function $f_{\ref{AABpath}}$.
\end{proposition}

The next lemma states that if there are many disjoint paths from vertex sets $A$ and $B$ each to a vertex set $X$, then we can also simultaneously find many disjoint paths from $A$ and $B$ to $X$.
\begin{lemma}[Bruhn and U. \cite{BU18}]\label{twotypespaths}
Let $G$ be a graph and let $A,B$ and $X$ be vertex sets in $G$.
If there is a set $\cQ$ of $2t$ disjoint $A$-$X$-paths and a set $\cR$ of $t$ disjoint $B$-$X$-paths in $G$, then there are $2t$ disjoint paths $P_1,\ldots,P_{2t}$ such that $P_1,\ldots P_t\subseteq \cQ$, which means they are $A$-$X$-paths, and $P_{t+1},\ldots,P_{2t}$ are $B$-$X$-paths with $P_i\subseteq \bigcup_{Q\in\cQ}Q \cup \bigcup_{R\in\cR}R$ for $i\in[2t]$. 
\end{lemma}

The way the proof works, is that we follow the paths in $\cR$ from $B$ until the first time they intersect paths in $\cQ$ and then we try to reroute the paths in $\cR$ along the paths in $\cQ$. If we cannot do that for all paths in $\cR$ at the same time (because the rerouted paths would intersect), we just follow some paths in $\cR$ a little bit longer until that is possible. This means that if we know that the paths in $\cR$, starting in $B$, intersect a path in $\cQ$ only after specific vertices, then the paths $P_{t+1},\ldots,P_{2t}$ will contain the subpaths of paths in $\cR$ from $B$ to these specific vertices. This will come in handy later.

\subsection{Walls and Linkages}
An \emph{elementary wall} $W$ of size $n\times m$ is the graph with vertex set 
$$\{v_{i,j}:i\in[n+1],j\in [2m+2]\}$$ 
\begin{center}
and edge set
\end{center}  
$$\{v_{i,j}v_{i',j'}: i=i' \textrm{ and } |j-j'|=1 \textrm{ or } i=i'-1 \textrm{ and } j=j' \textrm{ has the same parity as $i$}\}$$
where we then remove all vertices of degree $1$, that is $v_{1,2m+2}$ and if $n$ is odd $v_{n+1,1}$ or if $n$ is even $v_{n+1,2m+2}$. If $n=m$, we say that $W$ is a wall of size $n$ or an $n$-wall.
The \emph{$i^{th}$ row} of $W$ is the induced subgraph on the vertex set $\{v_{i,j}:j\in[2m+1]\}$ (without the degree $1$ vertices that were removed). We call the first row of $W$ the top row.
We can define a natural order in the top row where $v_{1,1}$ is the first vertex and for two vertices $v$ and $w$ in the top row, we denote by $v<w$ that $v$ comes before $w$.
There is a unique set of $n+1$ disjoint paths from the top row to the $(n+1)^{st}$ row. We call these paths \emph{columns} of $W$. We order the columns according to the order of their endvertex in the top row; the \emph{$j^{th}$ column} of $W$ is the column with the $j^{th}$ endvertex in the top row.
The \emph{distance} of the $i^{th}$ row and the $j^{th}$ row is $|j-i|$. Analogously, define the distance between two columns.
The \emph{nails} of $W$ are all the degree $2$ vertices in the top row of $W$ except for the first and last one, that is $v_{1,1}$ and $v_{1,2m+1}$. The \emph{branch vertices} of $W$ are the vertices of degree $3$. Every branch vertex lies in a row and a column and at most two lie in the same row and column. 

A \emph{wall} is a subdivision of an elementary wall. 
All the definitions for elementary walls except nails can be extended to walls in a natural way. We will give a natural extension of the definition for nails shortly.
A \emph{subdivided edge} in a wall with branch vertices $B$ is a $B$-path; it does not contain vertices of $B$ in its interior. 
Observe that a subdivided edge contains only vertices of degree $2$. We say that a wall is \emph{zero} if each subdivided edge is zero. In particular, this implies that all paths inside a zero wall between its branch vertices are zero. 

When we say that we \emph{remove a row or column}, we mean that we remove all vertices and edges that lie in said row or column but which do not lie in a column or row respectively. If we remove the first or last row or column, also iteratively remove all vertices of degree $1$. In any case, the resulting graph is a wall with one less row/column.

A \emph{subwall} $W_0$ of a wall $W$ is a wall such that each row of $W_0$ is a subset of one row of $W$ and each column of $W_0$ a subset of one column of $W$. The subwall $W_0$ is \emph{$t$-contained} in $W$ if it does not intersect the first and last $t$ rows and columns of $W$. If $W_0$ is a subwall that is at least $1$-contained in a wall $W$, we can define the nails of $W_0$ in a natural way: they are branch vertices in $W$. 
If this is not the case, choose any one degree $2$ vertex in each subdivided edge in the top row of $W_0$ and let these be the nails of $W_0$.

Let $W_0$ be a subwall of a wall $W$ and for any branch vertex $v$ of $W_0$, let $R_v$ be the row and $C_v$ the column of $W$ that contains $v$. A cycle $C$ in $W$ \emph{encapsulates $W_0$} if $W_0$ and $C$ are disjoint and if for every branch vertex $v$ of $W_0$, there are two disjoint paths from $v$ to $C$ in $C_v$ and two disjoint paths from $v$ to $C$ in $R_v$. 

A \emph{linkage} $L$ for a wall $W$ with nails $N$ is a set of disjoint $N$-paths in $G-(W-N)$. 
We say that a linkage is \emph{non-zero} if each path in it is non-zero.
For each linkage path with endvertices $v$ and $w$ such that $v<w$, we define $v$ as the \emph{left endvertex} and $w$ as the \emph{right endvertex} of that path.

Let $L=\{P_1,\ldots,P_n\}$ be a
linkage for a wall $W$ where $P_1,\ldots, P_n$ are ordered according to their left endvertex in the top row of $W$. For $i\in\{1,\ldots,n\}$ let $\ell_i$ and $r_i$ be the left and right endvertex of $P_i$ respectively. The linkage $L$ is \emph{in series} if for each $i<j$, it holds that $\ell_i<r_i<\ell_j<r_j$. It is \emph{nested} if for each $i<j$, $\ell_i<\ell_j<r_j<r_i$. Lastly, it is \emph{crossing} if for each $i<j$, $\ell_i<\ell_j<r_i<r_j$.
A linkage is \emph{pure} if it is either in series, nested or crossing.
Every linkage contains a subset that is a pure linkage.

\begin{lemma}[Huynh et al. \cite{HJW16}]\label{pureLink}
Let $L$ be a linkage for a wall. There is a subset $L'\subseteq L$ that is a pure linkage.
\end{lemma}

\subsection{Tangles}
We use the notion of tangles for the proof of the main theorem.
Let $G$ be a graph. We say that an ordered tuple $(C,D)$ is a \emph{separation} in $G$ if $C$ and $D$ are edge-disjoint subgraphs of $G$ such that $C\cup D=G$. The order of the separation $(C,D)$ is $|V(C\cap D)|$. A \emph{tangle} $\cT$ of order $t$ is a set of separations such that 
\begin{enumerate}
\item for every separation $(C,D)$ of order at most $t-1$, either $(C,D)\in \cT$ or $(D,C)\in\cT$ but not both
\item $V(C)\neq V(G)$ for all $(C,D)\in\cT$
\item $C_1\cup C_2\cup C_3\neq G$ for all $(C_1,D_1),(C_2,D_2),(C_3,D_3)\in\cT$
\end{enumerate}

We call $C$ and $D$ the \emph{sides} of $(C,D)$.
If $(C,D)\in\cT$, then we say that $D$ is the \emph{large side} and $C$ the \emph{small side} of the separation $(C,D)$ in $\cT$.
A tangle $\cT_0$ is a \emph{truncation} of a tangle $\cT$ if $\cT_0\subseteq \cT$.

Each wall $W$ of size $n$ induces a tangle $\cT_W$ of order $n+1$.
For each separation $(C,D)$ of order at most $n$, there is exactly one side that contains a whole row of $W$. We define that side as the large side, that is if $D$ contains a whole row of $W$, then $(C,D)\in\cT_W$ otherwise $(D,C)\in\cT_W$. That this really defines a tangle has been shown in \cite{RS91}. 
By applying the definitions, one can quite easily see that if $W_0$ is a subwall of $W$, then $\cT_{W_0}$ is a truncation of $\cT_W$. Robertson and Seymour have shown that a large tangle $\cT$ in a graph $G$ implies that there is a large wall in $G$. Furthermore, the tangle induced by the wall is a truncation of $\cT$.

\begin{theorem}[Robertson and Seymour \cite{RS91}]\label{tangleTOwall}
Let $G$ be a graph. For every positive integer $t$, there is an integer $f_{\ref{tangleTOwall}}(t)$ such that if $\cT$ is a tangle of order $f_{\ref{tangleTOwall}}(t)$ in $G$, then there is a wall $W$ of size $t$ in $G$ such that $\cT_W$ is a truncation of $\cT$.
\end{theorem}

Let $f:\N\to\N$ be some function. Assume that $f$ is not a hitting set function for the zero $A$-paths. It follows that there is an integer $k$ and a graph $G$ with a vertex set $A\subseteq V(G)$ such that $G$ contains neither $k$ disjoint zero $A$-paths nor a set of at most $f(k)$ vertices that intersects all those paths. If $k$ is chosen minimum, we say that $(G,A,k)$ is a minimal counterexample to $f$ being a hitting set function for the zero $A$-paths. If we assume that our main theorem is false, we can find a minimal counterexample $(G,A,k)$ to any function being a hitting set function. The next lemma states that if $f$ is chosen large enough, we can find a large tangle in $G-A$. 

\begin{lemma}[Bruhn and U. \cite{BU18}] \label{largetanglem}
Let $(G,A,k)$ be a minimal counterexample to $f:\N\to\N$ being a hitting set function for the zero $A$-paths such that $f$ satisfies $f(k)\geq 3g(k)+2f(k-1)+10$.
Then the graph $G-A$ admits a tangle \eptan\ of order $g(k)$ such that for each 
 separation $(C,D) \in \eptan$ every zero $A$-path has to intersect~$D-C$. 
\end{lemma} 

In \cite{BU18} this was shown for $A$-paths of length $0\ \textrm{mod}\ 4$. However, the proof is independent of the specific type of zero $A$-paths we are looking for. Thus, we may use it here as well.

\section{Wall Theorem}
In this section Theorem \ref{primeWall} will be proven.
We start with a generalization of a Theorem by Thomassen.
He showed that for any $m\in\N$, a sufficiently large wall also contains a large wall where all subdivided edges have length $0\ \m$. Note that this is just a zero wall where all edges are labelled with $1\in\Z_m$.

The proof of Thomassen actually does not necessitate this special case. We can use it in a much more general sense, that is any sufficiently large $\Gamma$-labelled wall contains a zero wall.
I want to stress here that the main part of the proof of this generalization is the same as the proof by Thomassen. It is included for the sake of completeness.

\begin{theorem}\label{ThomZero}
Let $G$ be a graph and assume that $\Gamma$ is finite such that $m=|\Gamma|$. For every positive integer $t$, there is an integer $f_{\ref{ThomZero}}(m,t)$ such that if $G$ contains a wall $W$ of size $f_{\ref{ThomZero}}(m,t)$, then $G$ also contains a zero wall $W_0$ of size $t$. Additionally, $\cT_{W_0}$ is a truncation of $\cT_W$.
\end{theorem}

\begin{proof}

Let $n_{3t+1}=3(t+1)$ and for $i\in\{0,\ldots,3t\}$, recursively define $n_i=3(n_{i+1}+1)m$.
Define $f_{\ref{ThomZero}}(m,t)=n_0$. Observe that $n_0$ depends on $m$ and $t$.

Let $x_{1,1},\ldots,x_{1,n_0+1}$ be the branch vertices in the first row $R_1$ of $W$ (ordered along the top row). For $i\in\{1,\ldots,n_0+1\}$, let $\ell_i$ be the weight of the subpath of $R_1$ from $v_{1,1}$ to $v_{1,i}$. As $\Gamma$ is finite, there is a subset $I_1\subseteq \{1,\ldots,n\}$ of size $\frac{n_0+1}{m}\geq n_1+1$ and a $\gamma\in\Gamma$ such that $\ell_i=\gamma$ for $i\in I_1$. For $i,j\in I_1$ with $i<j$, let $P_i,P_j$ be the paths from $x_{1,1}$ to $x_{1,i}$ and $x_{1,j}$ respectively. By choice of $I_1$, the paths $P_i$ and $P_j$ have weight $\gamma$ and since $i<j$, the path $P_j$ contains $x_{1,i}$. Now the subpath of $R_1$ between $x_{1,i}$ and $x_{1,j}$ is $P_j-P_i$, which means its weight is $\gamma-\gamma=0$.
This implies that all subpaths of $R_1$ between vertices $x_{1,i}$ with $i\in I_1$ are zero paths.

Now remove all columns of $W$ that do not contain a vertex $x_{1,i}$ for $i\in I_1$. 
Doing this, we obtain a subwall $W_1$ of $W$ of size $n_0\times n_1$ such that all paths in the first row of $W_1$ between branch vertices are zero. 

We can apply the exact same argument to the second row of $W_1$ to obtain a subwall $W_2$ of $W_1$ of size $n_0\times n_2$ where all paths in the first and all paths in the second row between branch vertices are zero. Continue doing this for the first $3t+1$ rows.
This yields a subwall $W_{3t+1}$ of $W$ of size $n_0\times n_{3t+1}$ such that all subpaths of rows between branch vertices are zero.

An elementary wall has vertex set $\{v_{i,j}:i\in[n+1],j\in[2n+2]\}$ (where two vertices were removed). As a wall is a subdivision of an elementary wall, we can define these vertices in a wall, too.
Let $W_0'=W_{3t+1}$.
For $j\in\{1,\ldots,\frac{n_{3t+1}}{3}\}$, we define the $j^{th}$ diagonal path as a path through subdivided edges from
$v_{1,2j-1}$ to $v_{2,2j-1}$ to $v_{2,2j}$ to $v_{3,2j}$ to $v_{3,2j+1}$ and so on. Since $j\leq \frac{n_{3t+1}}{3}$, it holds that all diagonal paths end in the last row of $W_{3t+1}$ and, thus, the $j^{th}$ diagonal path contains $n_0+1$ branch vertices of the form $v_{i,2j-1+i-1}$ for $i\in[n_0+1]$.
 
With the same argument as before, there is a set of indices $J_1$ of size $n_1+1$
such that the weight of each subpath of the first diagonal path between any two vertices of the form $v_{i_1,i_1}$ and $v_{i_2,i_2}$ is zero, if $i_1,i_2\in J_1$. 
Note that for $j=1$, a vertex of the form $v_{i,i}$ is also of the form $v_{i,2j-1+i-1}$.
We remove all rows of $W_0'$ that do not contain a vertex $v_{i,i}$ with $i\in J_1$. We obtain a subwall $W_1'$ of size $n_1\times n_{3t+1}$ such that all subpaths of its first diagonal path between branch vertices are zero. We use here that $W_1'$ is a subwall of $W_{3t+1}$ and, thus, the subdivided edges between vertices $v_{i,i-1}$ and $v_{i,i}$ are zero.

Apply this argument to the second diagonal path of $W_1'$ to obtain a subwall $W_2'$ such that all subpaths of the first and second diagonal path between branch vertices are zero. Continue this, until $W_{3t+1}'$ is constructed. Its size is $n_{3t+1}\times n_{3t+1}$.
All paths between branch vertices that only intersect diagonal paths or rows of $W_{3t+1}'$, are zero.

As before, we define $v_{i,j}'$ to be the vertices in $W_{3t+1}'$ of the underlying elementary wall.
Now we are able to define $W'$. The $i^{th}$ row of $W'$ is the $i^{th}$ diagonal path in $W_{3t+1}'$. Add the subpaths of rows of $W_{3t+1}'$ from $v_{i,j}$ to $v_{i,j+1}$ if $i+j+2\equiv 0\ \textrm{mod}\ 4$. To convince yourself that this really defines a wall, see Figure \ref{ThomWall}. There were $\frac{n_{3t+1}}{3}=t+1$ diagonal paths in $W_{3t+1}'$, which means there are $t+1$ rows in $W'$. Every two rows of $W_{3t+1}'$ contain a column of $W'$. Together this implies that the size of $W'$ is at least $t$. That $W'$ is a zero wall simply follows from the fact that it is subgraph of the rows and diagonal paths of $W_{3t+1}'$. 

\begin{figure}[ht]
\centering
\begin{tikzpicture}
\def\wallheight{8}
\def\brickheight{0.75}
\def\diags{5}
\def\diagx{17}
\tikzstyle{vx}=[thick,circle,inner sep=0.cm, minimum size=1.6mm, fill=white, draw=black]

\foreach \j in {0,...,\diags}{
	\foreach \i in {0,...,\diagx}{
		\ifnum\i=0
			\node[vx] (v\i\j) at (2*\j*\brickheight,0){};
		\fi
		\pgfmathsetmacro\m{\diagx-1};
		\pgfmathsetmacro\n{\diagx-2};
		\ifnum\j=0
			\ifnum\i=\m
				\breakforeach
			\fi
		\fi
		
		\ifnum\j=\diags
			\ifnum\i=\n
				\breakforeach
			\fi
		\fi		
		
		\ifnum\i>0
			\ifodd\i
				\pgfmathsetmacro\k{(\i-1)/2)};
				\pgfmathsetmacro\l{(\i+1)/2)};
				\node[vx] (v\i\j) at (\k*\brickheight+2*\j*\brickheight,-\l*\brickheight){};
			\else	
				\pgfmathsetmacro\k{\i/2)};
				\node[vx] (v\i\j) at (\k*\brickheight+2*\j*\brickheight,-\k*\brickheight){};
			\fi
		\fi
	}
}

\pgfmathsetmacro\diagxSub{\diagx-1};
\foreach \j in {0,...,\diags}{
	\foreach \i in {0,...,\diagxSub}{
		\pgfmathsetmacro\m{\diagx-3};
		\pgfmathsetmacro\n{\diagx-3};
		\ifnum\j=0
			\ifnum\i>\m
				\breakforeach
			\fi
		\fi	
		
		\ifnum\j=\diags
			\ifnum\i=\n
				\breakforeach
			\fi
		\fi			
		
		\pgfmathsetmacro\k{int(\i+1)};
		\draw[thick] (v\i\j) -- (v\k\j);					
	}
}

\pgfmathsetmacro\diagsSub{\diags-1};
\foreach \j in {0,...,\diagsSub}{
	\foreach \i in {1,...,\diagx}{
		\pgfmathsetmacro\k{int(\j+1)};
		\pgfmathsetmacro\l{int(\i-1)};
		\ifodd\j
			\pgfmathsetmacro\m{int(Mod(\i+2,4))};
			\ifnum\m=0
				\draw[dotted,thick] (v\i\j) -- (v\l\k);
			\fi
		\else
			\pgfmathsetmacro\m{int(Mod(\i,4))};
			\ifnum\m=0
				\draw[dotted,thick] (v\i\j) -- (v\l\k);
			\fi
		\fi		
	}
}

\pgfmathsetmacro\diagsSub{\diags-1};
\foreach \j in {0,...,\diagsSub}{
		\pgfmathsetmacro\k{int(\j+1)};
		\ifodd\j
			\draw[dotted,thick] (v\diagx\j) -- (v\diagx\k);
		\else
			\draw[dotted,thick] (v0\j) -- (v0\k);
		\fi				
}

\end{tikzpicture}
\caption{The rows of $W'$ are drawn in solid line and the subdivided edges between the rows of $W'$ in dotted lines.}\label{ThomWall}
\end{figure}
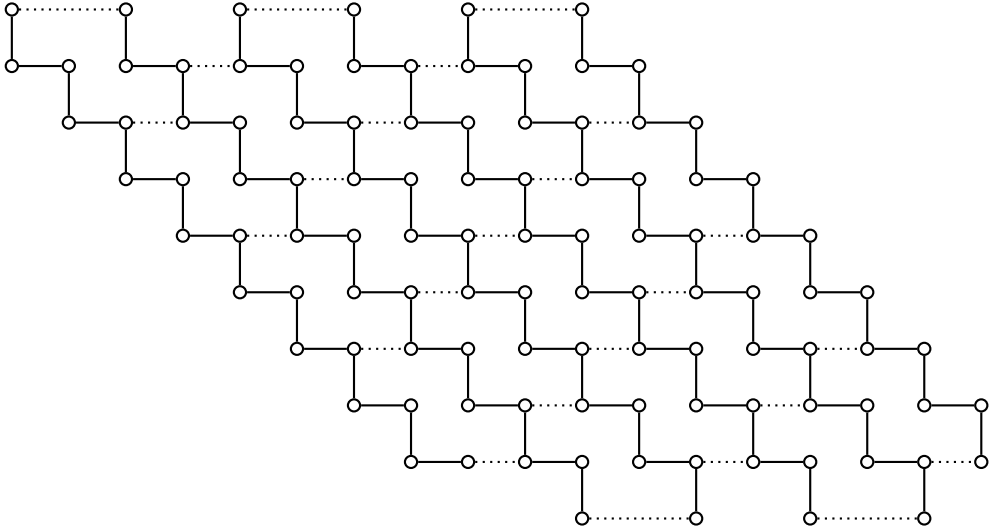

Now we still need to prove that $\cT_{W'}$ is a truncation of $\cT_W$.
Let $(C,D)$ be a separation of order at most $t$.
We need to show $(C,D)\in\cT_{W'}$ if and only if $(C,D)\in\cT_W$.
So let $(C,D)\in \cT_{W'}$ and suppose $(D,C)\in\cT_W$. 
Since the size of $W'$ is $t$, it contains $t+1$ rows, which implies that one row $R$ of $W'$ is disjoint from $V(C\cap D)$. As $(C,D)\in \cT_{W'}$, the row $R$ is contained in $D-C$.
With the same argument all but $t$ rows of $W$ lie in $C-D$.
It follows that $R$ intersects at most $t$ rows of $W$. However, by construction, every row of $W'$ intersects all but at most one row of $W_{3t+1}'$. These are more than $t+1$ rows and as $W_{3t+1}'$ is a subwall of $W$, the row $R$ also intersects more than $t+1$ rows of $W$. This is a contradiction. As the argument is symmetric in $C$ and $D$, this finishes the proof.
\end{proof}

The next lemma tells us that if all paths between vertices in a vertex set $B$ are zero, then the weight of all paths from $B$ to a fixed vertex that has three disjoint paths to $B$ is always the same.
This will also be useful when we prove the main theorem later.

\begin{lemma}\label{zeroM}
Let $G$ be a graph that contains a vertex set $B$
such that each $B$-path has weight zero. Let $M$ be the set of vertices $v$ that have three disjoint paths to $B\setminus \{v\}$. Then for each $v\in M$, there is a $\gamma\in\Gamma_2$ such that all paths from $v$ to $B$ have weight $\gamma$. Moreover, if for $v,w\in M$ any path with endvertices paths from $v$ and $w$ to $B$ are zero, then also all paths with endvertices $v$ and $w$ are zero.
\end{lemma}

\begin{proof}
First a general observation:
\begin{equation}\label{PiZero}
\begin{minipage}[c]{0.8\textwidth}\em
Let $u$ be a vertex of $G$ with three disjoint paths $P_1,P_2$ and $P_3$ to vertices in $B\setminus \{u\}$.
There is a $\gamma\in\Gamma_2$ such that each path $P_i$ has weight $\gamma$.
\end{minipage}\ignorespacesafterend 
\end{equation}

Let $x$ be the weight of $P_1$. The paths $P_1\cup P_2$ and $P_1\cup P_3$ are $B$-paths, which means they are zero. Hence, the weight of $P_2$ and $P_3$ is $-x$. Now the weight of $P_2\cup P_3$ is $-2x$, which again has to be zero. Thus, $-x$ is its own inverse, which implies that $x=-x$ and, therefore, $x\in \Gamma_2$.
This proves (\ref{PiZero}).

Let $v\in M$. 
Let $P_1,P_2$ and $P_3$ be three disjoint paths from $v$ to $B\setminus \{v\}$. Let $w$ be some vertex such that all paths from $w$ to $B$ are zero and let $P$ be a path from $v$ to $w$ such that $P$ intersects each path $P_1,P_2,P_3$. We claim that :
\begin{equation}\label{rerute}
\begin{minipage}[c]{0.8\textwidth}\em
there is a path $P'$ from $v$ to $B$ of the same weight as $P$ that contains less edges outside of $E(P_1\cup P_2\cup P_3)$.
\end{minipage}\ignorespacesafterend 
\end{equation}

For $i\in\{1,2,3\}$, let $u_i$ be the vertex in $P_i\cap P$ that is closest to $w_i$ on $P_i$. We may assume that, starting in $v$, $P$ passes through $u_1,u_2$ and $u_3$ in this order; otherwise swap the indices of the paths $P_1,P_2$ and $P_3$. Now $w_1P_1u_1Pu_2P_2w_2$ is a $B$-path and, thus, a zero path and $w_1P_1u_1Pw$ is a zero path by choice of $w$. It follows that the weight of $u_2P_2w_2$ and the weight of $u_2Pw$ is the same. In $P$, we replace $u_2Pw$ by $u_2P_2w_2$ and obtain a path $P'$ that has the same weight as $P$ and which is a path from $v$ to $B$.
Moreover, since $P_2$ and $P_3$ are disjoint outside the vertex $v$, which is already contained in $P$, the path $u_2Pu_3$ contains an edge outside of $P_1\cup P_2 \cup P_3$. This implies that $P'$ contains fewer edges outside of $P_1\cup P_2 \cup P_3$ than $P$, which proves (\ref{rerute}).

Now let $v$ be a vertex with three disjoint paths $P_1,P_2$ and $P_3$ to $B\setminus \{v\}$. By (\ref{PiZero}), there is a $\gamma\in\Gamma_2$ such that each path $P_i$ has weight $\gamma$.
Let $W$ be the set of vertices of $w$ such that all paths from $w$ to $B$ are zero. If we can show that all paths from $v$ to $W$ have weight $\gamma$, then this finishes the proof. 

Suppose otherwise. Let $P$ be the path from $v$ to $W$ that minimizes $|E(P\setminus (P_1\cup P_2\cup P_3))|$ such that its weight is not $\gamma$. If $P$ is disjoint from a path $P_i$, then $P\cup P_i$ is a path from $w$ to $B$, which means it is a zero path. As the weight of $P_i$ is $\gamma$ and as $\gamma\in\Gamma_2$, also the weight of $P$ is $\gamma$. This is a contradiction.

Suppose that $P$ intersects each path $P_i$. By (\ref{rerute}), there is a path $P'$ from $v$ to $B$ that has the same weight as $P$ but intersects fewer edges outside of $E(P_1\cup P_2\cup P_3)$ than $P$. Since $B\subset W$, this is a contradiction. 
\end{proof}

With this knowledge we can prove Theorem \ref{primeWall}.
The idea is very simple: Start with a sufficiently large wall $W$ in a graph $G$ and, with Theorem \ref{ThomZero}, find a large zero wall $W'$ inside of $W$. Then we apply Wollan's theorem to the set $B$ of branch vertices of $W'$ to either find a bounded set $X$ such that all $B$-paths in $G-X$ are zero or we find many non-zero $B$-paths that can be made into a non-zero linkage for a subwall of $W'$. In either case we are done.

\primeWall*

\begin{proof}
For fixed $s$, let $x=f_{\ref{Wollan}}(s^3)$ and $p=2(2(3x+1)^2+3x)^2+x$. Now set $g_{\ref{primeWall}}(s)=g_{\ref{primeWall}}(p)$ and $f_{\ref{primeWall}}(m,r,s)=f_{\ref{ThomZero}}(m, \max\{2g_{\ref{primeWall}}(p)r, \max\{r,4f_{\ref{Wollan}}(s^3)\}+10f_{\ref{Wollan}}(s^3)\})$. 
By Theorem \ref{ThomZero} there is a zero wall $W'$ of size at least $\max\{2g_{\ref{primeWall}}(p)r, \max\{r,4f_{\ref{Wollan}}(s^3)\}+10f_{\ref{Wollan}}(s^3)\}$ in $G$. Furthermore, the tangle $\cT_{W'}$ is a truncation of the tangle induced by $W$. Note that $W_0$ will be a subwall of $W'$, which means that $\cT_{W_0}$ will be a truncation of $\cT_W$.

Apply Theorem \ref{Wollan} to the set $B$ of branch vertices of $W'$. There are either $p$ disjoint non-zero $B$-paths in $G$ or a set $X$ of at most $g_{\ref{primeWall}}(p)$ vertices that intersects all these paths. In the latter case there is a subwall $W_0$ of $W'$ that is disjoint from $X$.
This is because $W'$ contains more than $2g_{\ref{primeWall}}(p)$ disjoint subwalls of size $r$, each being a zero wall. One of them has to be disjoint from $X$.
Then, since all branch vertices of $W_0$ are also branch vertices of $W'$, all paths in $G-X$ between branch vertices of $W_0$ have length zero. So we are done in this case.

Now we may assume that there is a set $\cP$ of $p$ disjoint non-zero $B$-paths. We choose $\cP$ such that the number of edges outside of $E(W)$ is minimized, that is $\sum_{P\in\cP}|E(P)\setminus E(W)|$ is minimized.
Remember that, by definition, any non-zero $B$-path does not contain any vertices of $B$ in its interior. 
We claim:

\begin{equation}\label{fewEdges}
\begin{minipage}[c]{0.8\textwidth}\em
the paths in $\cP$ cannot intersect more than $10|\cP|$ subdivided edges of $W$.
\end{minipage}\ignorespacesafterend 
\end{equation}

Suppose otherwise. For each subdivided edge $e$ of $W$ that is being intersected by a path in $\cP$, we find a path $P_e\in\cP$ that is closest on $e$ to one of the two endvertices of $e$ (there is one such path for each endvertex of a subdivided edge but it may be the same). Any path that contains a branch vertex $v$ is closest to $v$ on any subdivided edge that contains $v$.
By pigeonhole principle, there is one path $P\in\cP$ that is closest to endvertices on $10$ different subdivided edges of $W'$. Let $E^*$ be the set of these subdivided edges.
As the degree of branch vertices is $3$, 
at most $6$ subdivided edges in $E^*$ contain an endvertex of $P$;
remove these from $E^*$.
The remaining subdivided edges in $E^*$ are disjoint from all endvertices of paths in $\cP$. Note that the size of $E^*$ is at least $4$. 

Again because the degree of branch vertices is $3$, there are two distinct branch vertices $v_1$ and $v_2$ that $P$ is closest to on distinct subdivided edges $e_1,e_2\ E^*$ respectively.
Let $a,b$ be the endvertices of $P$ and for $j\in\{1,2\}$, let $w_j$ be the intersection of $P$ and $e_j$ that is closest to $v_j$. We may assume that starting in $a$ the vertex $w_1$ comes before $w_2$ on $P$. Otherwise swap the roles of $a$ and $b$.

We want to show that we can reroute $P$ along $e_1$ or $e_2$  
to find a non-zero $B$-path $P'$ that has fewer intersections with $E(G)\setminus E(W)$ than $P$. This implies that replacing $P$ by $P'$ would yield a better choice for $\cP$ (assuming that $P'$ is still disjoint from all other paths in $\cP$). So let $\gamma_1$ be the weight of $aPw_1$, $\gamma_2$ the weight of $w_1Pw_2$ and $\gamma_3$ the weight of $w_2Pb$. For $j\in\{1,2\}$, let $\alpha_j$ be the weight of the path $w_je_jv_j$. See Figure \ref{fewEdgeDraw} for a drawing of all the important vertices and weights.
Note that the weight of $P$ is $\gamma_1+\gamma_2+\gamma_3$
and as $P$ is non-zero, it follows that $\gamma_1+\gamma_2+\gamma_3\neq 0$.

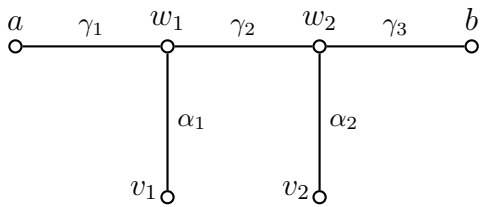
\begin{figure}[ht]
\centering
\begin{tikzpicture}
\def\distance{2}
\tikzstyle{vx}=[thick,circle,inner sep=0.cm, minimum size=1.6mm, fill=white, draw=black]

\node[vx,label=\large{$a$}] (a) at (0,0) {};
\node[vx,label=\large{$b$}] (b) at (3*\distance,0) {};
\node[vx,label=\large{$w_1$}] (w1) at (\distance,0) {};
\node[vx,label=\large{$w_2$}] (w2) at (2*\distance,0) {};
\node[vx,label={[left]\large{$v_1$}}] (v1) at (\distance,-\distance) {};
\node[vx,label={[left]\large{$v_2$}}] (v2) at (2*\distance,-\distance) {};

\path[-,draw,thick]
 	(a) edge node[auto] {$\gamma_1$} (w1)
 	(w1) edge node[auto] {$\gamma_2$} (w2)
	(w2) edge node[auto] {$\gamma_3$} (b)
	(w1) edge node[auto] {$\alpha_1$} (v1)
	(w2) edge node[auto] {$\alpha_2$} (v2);


\end{tikzpicture}
\caption{A drawing of $P$ and the subpath of $e_j$ from $w_j$ to $v_j$ for $j\in[2]$.}\label{fewEdgeDraw}
\end{figure}

Any path from $a$ or $b$ to the interior of $e_1$ or $e_2$ either contains an endvertex of $e_1$ or $e_2$ or an edge outside of $W$.
Since $a$ and $b$ are not endvertices of $e_1$ and $e_2$ and because $P$ does not contain branch vertices in its interior, it follows that
the paths $aPw_1$ and $bPw_2$ contain an edge of $G-W'$.
With this observation we can construct three $B$-paths that contain fewer edges outside of $E(W')$ than $P$.
These are $aPw_2e_2v_2$ and $v_1e_1w_1Pb$ and $v_1e_1w_1Pw_2e_2v_2$. Since $P$ was the closest path on $e_1$ and $e_2$ and since $w_1$ and $w_2$ were the closest vertices to $v_1$ and $v_2$ on $e_1$ and $e_2$ respectively, the constructed paths do not intersect any paths in $\cP$ except $P$.
If one of these paths was non-zero, we could replace $P$ by this path, which would yield a contradiction.
Thus, we obtain the following equations:

$$\gamma_1+\gamma_2+\alpha_2=0$$ 
$$\alpha_1+\gamma_2+\alpha_2=0$$ 
$$\alpha_1+\gamma_2+\gamma_3=0$$ 

From the first two equations we obtain $\alpha_1=\gamma_1$ and from the last two equations we obtain $\alpha_2=\gamma_3$. It follows that $\alpha_1+\gamma_2+\alpha_2=\gamma_1+\gamma_2+\gamma_3\neq 0$. This is a contradiction, which proves (\ref{fewEdges}).

By (\ref{fewEdges}) and since the size of $W'$ is larger than $20p\max\{r,4f_{\ref{Wollan}}(s^3)\}+10f_{\ref{Wollan}}(s^3)$, there is a subwall $W_0$ of $W'$ of size $\max\{r,4f_{\ref{Wollan}}(s^3)\}$ that is $4f_{\ref{Wollan}}(s^3)$-contained and that is disjoint from all paths in $\cP$
We claim:
\begin{equation}\label{manyDisjoint}
\begin{minipage}[c]{0.8\textwidth}\em
there are $s^3$ disjoint non-zero $N_0$-paths in $G-(W_0-N_0)$.
\end{minipage}\ignorespacesafterend 
\end{equation}

Note that this is a non-zero linkage for $W_0$.
Applying Lemma \ref{pureLink} here, implies that we can find a pure non-zero linkage of size $s$ for $W_0$ and are done if this is true. Suppose it is not. By Wollan's theorem, there is a set $X$ of at most $x=f_{\ref{Wollan}}(s^3)$ vertices that intersects all these paths. At least 

For each path in $\cP$, arbitrarily choose one endvertex to be the first endvertex and the other one to be the second endvertex.
In $\cP$ there are $2(2(3x+1)^2+3x)^2$ many paths that are disjoint from $X$. By pigeonhole principle and as at most two vertices may lie in the same row and column, there are $2(3x+1)^2+3x$ paths such that their first endvertices either lie in different rows or lie in different columns. From these paths at least $2(3x+1)^2$ lie in a row or column such that this row or column and the rows and columns at distance $1$ are disjoint from $X$. Applying the same argument to the second endvertices, yields a path $P$ such that both endvertices lie in a row or column such that this row or column and the rows or columns at distance $1$ are disjoint from $X$.

As $W_0$ is $4f_{\ref{Wollan}}(s^3)$-contained, there are three cycles that encapsulate $W_0$ and that are disjoint from $X$. Each cycle $C_i$ intersects all but at most $8f_{\ref{Wollan}}(s^3)$ rows and columns of $W'$, which means it intersects more than $f_{\ref{Wollan}}(s^3)$ Since the size of $W_0$ is $4f_{\ref{Wollan}}(s^3)$, there are three columns of $W$ that are disjoint from $X$ and that intersect $N_0$. All of these columns contain a path from $N_0$ to each $C_i$ that is disjoint from $W_0-N_0$ because each $C_i$ is an encapsulating cycles.
Lastly there are three rows and three columns that intersect each $C_i$ and are disjoint from $W_0$ and $X$. Let $v$ be an endvertex of $P$. As the row or column in which $v$ lies is disjoint from $X$, it is not that hard to see that we can find three disjoint paths from $v$ to $N_0$ in $W'-(W_0-N_0)-X$. As $W'$ is a zero wall, there is a zero path from $v$ to $N_0$. Now applying Lemma \ref{zeroM} implies that all paths from $v$ to $N_0$ are zero. The same is true for the other endvertex of $P$. The second part of Lemma \ref{zeroM} implies that all paths in $G-(W_0-N_0)-X$ between the two endvertices of $P$ are zero. This is a contradiction to $P$ being a non-zero path, which proves (\ref{manyDisjoint}).
\end{proof}

\section{Necessity of Theorem \ref{mainThm}}
In this chapter I will construct counterexamples for the $A$-paths of weight $\gamma$ if $\Gamma$ is infinite or if there are $x,y\in \Gamma$ with $y\neq 0$ such that for all integers $n\in\N$ it holds that $2x+ny\neq \gamma$.
The counterexamples will be constructed in the following way: Take a grid $W$ of size $n$ (or an elementary wall for that matter) and label all edges that are not in the top row with $0$ and then depending on $\Gamma$ label alld edges in the top row with some $\alpha_1\in\Gamma$.
Then add $2n$ vertices $A$ such that half of them are adjacent to the vertices on the left side and the other half adjacent to the vertices on the right side. Label all edges to the left side with some $\alpha_2$ and all edges to the right with some $\alpha_3$ (see Figure \ref{counterFigure}).

The idea is to show that any $A$-path of weight $\gamma$ has to start with an edge to the left side from $A$ and end with an edge to the right side from $A$ while picking up at least one edge in the top row of $W$. If additionally, such an $A$-path exists with the exception that this path picks up at most one edge in the top row, then there are no two disjoint $A$-paths of weight $\gamma$ and there is no set of size smaller than $\frac{n}{10}$ that intersects all these paths.
If this is indeed the case, then the $A$-paths of weight $\gamma$ do not have the \EPP.
This has been done multiple times before, for example in \cite{BHJ18b} or \cite{BU18}.
For this reason, I will not prove this again.

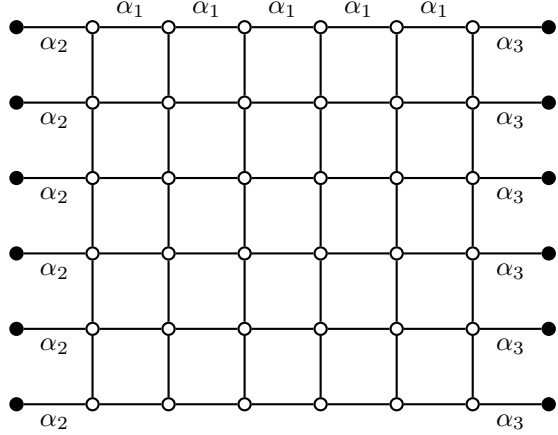
\begin{figure}[ht]
\centering
\begin{tikzpicture}
\def\size{6}
\def\gridheight{1}
\tikzstyle{vx}=[thick,circle,inner sep=0.cm, minimum size=1.6mm, fill=white, draw=black]
\tikzstyle{A}=[thick,circle,inner sep=0.cm, minimum size=1.6mm, fill=black, draw=black]

\foreach \i in {1,...,\size}{
	\foreach \j in {1,...,\size}{
			\node[vx] (v\i\j) at (\j*\gridheight,\i*\gridheight){};
	}
}

\pgfmathsetmacro\m{\size+1};	
\foreach \i in {1,...,\size}{
	\node[A] (u\i) at (0,\i*\gridheight){};
	\node[A] (w\i) at (\m*\gridheight, \i*\gridheight){};
}		

\pgfmathsetmacro\sizeminus{int(\size-1)};
\foreach \i in {1,...,\sizeminus}{
	\pgfmathsetmacro\iplus{int(\i+1)};
	\foreach \j in {1,...,\size}{
		\pgfmathsetmacro\jplus{int(\j+1)};
		\ifnum \i<\size
			\draw[thick] (v\i\j) -- (v\iplus\j);
		\fi
		\ifnum \j < \size
 			\draw[thick] (v\i\j) -- (v\i\jplus);
 		\fi
	}
}

\foreach \j in {1,...,\sizeminus}{
	\pgfmathsetmacro\jplus{int(\j+1)};
	\path[-,draw,thick] (v\size\j) edge node[auto] {$\alpha_1$} (v\size\jplus);
}


\foreach \i in {1,...,\size}{
	\path[-,draw,thick] (u\i) edge node[below] {$\alpha_2$} (v\i1);
	\path[-,draw,thick] (w\i) edge node[below] {$\alpha_3$} (v\i\size);
}

\end{tikzpicture}
\caption{The counterexample of size $6$. All unlabelled edges are labelled with $0$ and the coloured vertices are the vertices of $A$.}\label{counterFigure}
\end{figure}

We will show the following:

\begin{lemma}
Let $\gamma\in\Gamma$. If $\Gamma$ is infinite or if there are $x,y\in\Gamma$ with $y\neq 0$ such that for all integers $n$ the equation $2x+ny=\gamma$ is never satisfied, then the $A$-paths of weight $\gamma$ do not have the \EPP.
\end{lemma}

\begin{proof}
Suppose there are $x,y\in\Gamma$ with $y\neq 0$ such that for all integers $n$ the equation $2x+ny=\gamma$ is never satisfied. We set $\alpha_1=y$,$\alpha_2=x$ and $\alpha_3=-x-y+\gamma$ in the counterexample graph we defined before.
Observe that a path $P$ that starts on the left side, goes to the right side and picks up one edge in the top row is an $A$-path of weight $\gamma$.

Now let $P$ be any $A$-path. If $P$ goes from the left side to the right side without intersecting the top row, then its weight is $\gamma-y$ which is not $\gamma$ since $y\neq 0$. If both endvertices of $P$ are on the left side, then its weight is $2x+ny$ for some positive integer $n$. By assumption, this is distinct from $\gamma$ for any $n$. Lastly suppose that both endvertices of $P$ are on the right side, then the weight of $P$ is $-2x+(n-2)y+2\gamma$ for some positive integer $n$. Suppose that the weight of $P$ is $\gamma$. It follows that $\gamma=2x+(-n+2)y$, which is a contradiction. It follows that if $P$ is an $A$-path of weight $\gamma$, then it goes from the left to the right side and picks up at least one edge in the top row. With the remark before the lemma, we are done.

From now on assume that $\Gamma$ is infinite.
Suppose that $\Gamma_2$ is infinite, too, and that $\gamma=0$.
Let $(x_1,x_2,\ldots)$ be an enumeration of $\Gamma_2$.
We construct a graph $G_n$ as follows:
Take the counterexample of size $n$ as above with $\alpha_1=0$,
but differing from the construction before, we label the edges on the left side with $x_1,\ldots,x_{n}$ from top to bottom and on the right side with $x_{n},\ldots,x_1$ from top to bottom. 
Any zero $A$-path starts with an edge of weight $x_i$ and because all edges in the grid are labelled with $0$ it has to end with the edge with weight $x_i$ on the other side. Because of the inverted order $x_1,\ldots,x_{n}$ on both sides, no two disjoint zero $A$-paths exist in this graph. On the other hand, no set of at most $\frac{n}{10}$ can intersect all zero $A$-paths. We use here that there is a zero $A$-path that picks up at most one edge in the top row of the grid, therefore, the note before the lemma can be applied here as well.
This means that in this case the $A$-paths of weight $\gamma$ do not have the \EPP.

In any other case we can show that $\Gamma$ contains elements $x$ and $y\neq 0$ such that $2x+ny=\gamma$ is never satisfied for any integer $n$.

If $\Gamma_2$ is infinite and $\gamma\neq 0$, then let $x_1,x_2\in\Gamma_2$ such that $x_i\neq \gamma$. It follows that $2x_1+nx_2\in\{0,x_2\}$ and by choice, this never yields $\gamma$.

Now we may assume that $\Gamma$ contains infinitely many elements of order larger than $2$. Suppose there are two element $x',y'$ of infinite order, that is $nx'\neq 0$ for all $n\in\Z\setminus\{0\}$ and the same for $y'$. It follows that there is an $n_0$ such that $ny\neq -2x+\gamma$ for all $|n|\geq n_0$. If $2x'=\gamma$, we set $x=2x'$ otherwise we set $x=x'$. In any case, it holds that $2x'\neq \gamma$. We set $y=n_0y'$ and now the equation $2x+ny=\gamma$ is not true for $n=0$ because of the choice of $x$ and for any $n\neq 0$ it holds that $2x+ny=2x+nn_0y\neq \gamma$ by choice of $n_0$. Thus, $2x+ny\neq \gamma$ for all $n$.

Lastly, we may assume that $\Gamma$ contains two elements $x,y$ of finite order but of order at least $3$ such that the subgroups of $\Gamma$ induced by $x$ and $y$ intersect only in $0$. (The subgroup induced by an element $x$ are all the elements $x'$ that are multiples of $x$, that is $x'=nx$ for an integer $n$.) 
Additionally, if $\gamma\neq 0$, choose $x$ such that the subgroup induced by $x$ does not contain $\gamma$.
The reason why we find such $x$ and $y$ is because their order is finite. For any element in the subgroup induced by $x$ or $y$, also its inverse is in the same subgroup. If $\gamma=0$, then $2x+ny=0$ implies that the inverse of $2x$ is $-ny$. This is only possible if $2x=-ny=0$, which contradicts the fact that the order of $x$ is larger than $2$. If $\gamma\neq 0$, then $2x+nx\neq \gamma$ for any $n$ by choice of $x$. Altogether, if $\Gamma$ is infinite, then for any $\gamma\in\Gamma$, the $A$-paths of weight $\gamma$ do not have the \EPP.
\end{proof}

\section{Sufficiency of Theorem~\ref{mainThm}}
Before we start with the proof, we need two more lemmas on how to obtain zero $A$-paths in some special cases.

\subsection{Constructing zero $A$-paths}
In the following, let $\cT_{EP}$ be a tangle of sufficiently large size such that for $(C,D)\in\cT_{EP}$, any zero $A$-path has to intersect $D-C$.

Let $G$ be a graph and $A\subseteq V(G)$.
Let $W$ be a wall in $G$ with nails $N$ that is disjoint from $A$.
We say that a set of disjoint paths $\cP$ from $A$ to $N$ \emph{nicely links} $A$ to $W$ if the paths in $\cP$ are disjoint from $W-N$.

The following lemma has been proven in \cite{BU18},
although this is a quite general lemma, which means it might have appeared elsewhere, too.
\begin{lemma}[Bruhn and U. \cite{BU18}]\label{connectA}
Let $G$ be a graph and let $A\subseteq V(G)$.
Let $W$ be a wall such that $\cT_W$ is a truncation of $\cT_{EP}$.
For any positive integers $r$ and $t$ there is an integer $f_{\ref{connectA}}(r,t)$ such that if the size of $W$ is at least $f_{\ref{connectA}}(r,t)$, then there is a subwall $W_0$ of $W$ of size at least $r$ and a set $\cP$ of size $t$ that nicely links $A$ to $W_0$.
\end{lemma}

Clearly, if we choose the size of $W$ to be $f_{\ref{connectA}}(r,t)+2n$, then $W_0$ may be chosen to be $n$-contained in $W$.

The next lemma tells us that we can make paths that nicely link $A$ to a wall $W$ and a linkage $W$ disjoint. The outer cycle of a wall is the union of the first and last row and the first and last column. 

\begin{lemma}[Bruhn and U. \cite{BU18}]\label{disjointPathLink}
Let $t\geq 2$ be a positive integer, and let {$H$} be a graph, let $A\subseteq {V(H)}$ and let
$W$ be a wall.
Let $\mathcal P$ be a set of $3t$ disjoint paths that nicely links $A$ to $W$,
and let $L$ be a linkage for $W$ of size $6t$.
Then there is a set $\mathcal P'$ of size $t$ that nicely links $A$ to $W$,
and a subset $L'$ of $L$ of size $t$ such that the paths in $\mathcal P'\cup L'$
are pairwise disjoint. Moreover, there is an edge $e$ in the outer cycle $C$ of $W$
such that the endvertices of the paths in $\mathcal P'$ preceed the endvertices of the paths in $L'$
in the path $C-e$.
\end{lemma}

There are basically two outcomes here: All endvertices of $\cP'$ come before or after all endvertices of $L'$ in the top row of $W$ (this is almost the same since walls are essentially symmetrical) or they come inbetween two consecutice endvertices of $L'$.

When we apply Theorem \ref{primeWall}, one of the outcomes is a zero wall with a non-zero linkage. In the next lemma we deal with that case.

\begin{lemma}\label{nonzeroLinkCase}
Let $\Gamma$ be finite with $m=|\Gamma|$ and such that for any $x,y\in\Gamma$ with $y\neq 0$, there is an $n\in\Z$ such that $2x+ny=0$.
Let $G$ be a graph and let $A\subseteq V(G)$.
For any positive integer $k$, there are integers $f_{\ref{nonzeroLinkCase}}(m,k)$ and $g_{\ref{nonzeroLinkCase}}(m,k)$ such that if $W$ is a zero wall of size $f_{\ref{nonzeroLinkCase}}(m,k)$ in $G-A$ such that $\cT_{W}$ is a truncation of $\cT_{EP}$  and $L'$ is a pure non-zero linkage $L$ of size $g_{\ref{nonzeroLinkCase}}(m,k)$ for $W$, then there are $k$ disjoint zero $A$-paths in $G$.
\end{lemma}

\begin{proof}
Set $f_{\ref{nonzeroLinkCase}}(m,k)=f_{\ref{connectA}}(50m^2k,6m^2k)+100mk$ and $g_{\ref{nonzeroLinkCase}}(m,k)=12m^2k$.
We begin by applying Lemma \ref{connectA} to $W$ to obtain a subwall $W_0$ of $W$ of size at least $50m^2k$ and a set of paths $\cP$ of size $6m^2k$ that nicely link $A$ to $W_0$. Additionally, we may assume that $W_0$ is $50mk$-contained in $W$. Let $N_0$ be the nails of $W_0$.

At least $\frac{|\cP|}{m}=6mk$ paths in $\cP$ have the same weight $x$; remove all paths from $\cP$ but $6mk$ paths of weight $x$.
Also in $L$, at least $\frac{|L|}{m}=12mk$ many paths have the same non-zero weight, say $y$; remove all paths in $L$ but $12mk$ paths of weight $y$. By assumption, there is an $n$ such that $2x+ny=0$. Since the size of $\Gamma$ is $m$, the order of $y$ is at most $m$, which means there is an $1\leq \ell\leq m$ such that $\ell y=0$. This implies that we may choose $1\leq n\leq m$ so that $2x+ny=0$; we do that.

Since $W_0$ is $50mk$-contained, there is a set of disjoint paths in $W-(W_0-N_0)$ that connects the endvertices of $L$ to $N_0$ while retaining the order of the endvertices in the top row of $W$. This yields a pure linkage $L'$ for $W_0$ of size $12mk$ that has the same type as $L$. Moreover, as $W$ is a zero wall, all paths in $L'$ have the same weight as the paths in $L$, that is $y$.

Now apply Lemma \ref{disjointPathLink} to $\cP$ and $L$ to obtain a set $\cP'$ of size at least $2k$ that nicely links $A$ to $W_0$ and a set $L'\subseteq L$ of size at least $mk$ such that all paths in $\cP'\cup L'$ are pairwise disjoint. Furthermore, there is an edge $e$ in the outer cycle $C$ such that all endvertices of $\cP'$ come before $L'$ in $C-e$. As we have mentioned before, there are essentially two outcomes here: All endvertices of $\cP'$ come before or after all endvertices of $L'$ or inbetween two consecutive endvertices. The idea now is to take a path from $\cP'$, then pick up $n$ of the linkage path and end with another path in $\cP'$. Since $W_0$ is a zero wall, the weight of this path is $2x+ny=0$, which means we found a zero $A$-path.

Assume that the endvertices of $\cP'$ come before the endvertices of $L'$ in the top row of $W_0$ and assume that $L$ is is in series.
All the other cases can be done by slightly adjusting the argument.
Let $P_1,\ldots,P_{2t}$ be the paths in $\cP'$ ordered according to their endvertices in the top row of $W_0$ and let $L_1,\ldots,L_{mk}$ be the paths in $L'$ ordered according to their left endvertex. 
We construct a set $\cQ$ of disjoint paths from the nails of $W_0$ to the last row. Starting in a nail, follow the top row of $W_0$ to the left until we meet the first branch vertex, then follow the column in which this branch vertex lies to the last row.
Now for $i\in [k]$, connect the endvertices of $P_{2i}$ and $L_{(t-i)m+1}$ through the two paths in $\cQ$ that start in these endvertices and the $(2(t-i)+1)^{st}$ row. Connect the right endvertex of $L_{(t-i)m+1}$ to the left endvertex of $L_{(t-i)m+2}$ through the top row of $W_0$ and then the right endvertex of $L_{(t-i)m+2}$ to the left endvertex of $L_{(t-i)m+3}$. Continue this until $n$ linkage paths are connected to $P_{2i}$. Lastly, connect the right endvertex of $L_{(t-i)m+n}$ to the endvertex of $P_{2i-1}$ through the two paths in $\cQ$ that share their endvertices and the $(2i-1)^st$ row of $W_0$. This gives us an $A$-path $P_i'$.

The paths $P_1',\ldots,P_k'$ are disjoint by construction.
The length of each path $P_i'$ is $2x+ny$ which is zero. Therefore, we are done.

\begin{figure}[ht]
\centering
\begin{tikzpicture}
\def\size{4}
\def\height{0.8}
\tikzstyle{vx}=[thick,circle,inner sep=0.cm, minimum size=1.6mm, fill=white, draw=black]
\tikzstyle{A}=[thick,circle,inner sep=0.cm, minimum size=1.6mm, fill=black, draw=black]

\draw[gray,dashed] (0,0)--(10.5,0);

\foreach \i in {1,...,\size}{
		\node[A] (a\i) at (\i*\height,2){};
		\node[vx] (n\i) at (\i*\height,0){};
		\draw[gray,dashed] (a\i) -- (n\i);
}

\foreach \i in {1,...,\size}{
		\pgfmathsetmacro\j{int(\i)}
		\pgfmathsetmacro\k{int(\i)}
		\pgfmathsetmacro\iplus{(2*\j)-1+\size)*\height};
		\pgfmathsetmacro\iminus{(2*\k+\size)*\height};
		\node[vx] (l\i) at (\iplus,0){};
		\node[vx] (r\i) at (\iminus,0){};
		\draw[gray,dashed] (l\i) to[out=90,in=90] (r\i);
}

\draw[thick] (a2) -- (n2) -- (1.2,0) -- (1.2,-0.8) -- (1.6,-0.8) -- (1.6,-1.6) -- (1.2,-1.6) -- (1.2,-2.4);

\draw[thick] (a3) -- (n3) -- (2,0) -- (2,-0.8) -- (2.4,-0.8) -- (2.4,-1.6);

\draw[thick] (l2) -- (5.2,0) -- (5.2,-0.8) -- (5.6,-0.8) -- (5.6,-1.6);

\draw[thick] (r3) -- (7.6,0) -- (7.6,-0.8) -- (8,-0.8) -- (8,-1.6) -- (7.6,-1.6) -- (7.6,-2.4);

\draw[thick] (1.2,-2.4) -- (7.6,-2.4);
\draw[thick] (2.4,-1.6) -- (5.6,-1.6);
\draw[thick] (l2) to[out=90,in=90] (r2);
\draw[thick] (l3) to[out=90,in=90] (r3);
\draw[thick] (r2) -- (l3);

\end{tikzpicture}
\caption{This is how we construct the zero $A$-paths.}\label{zeroA}
\end{figure}
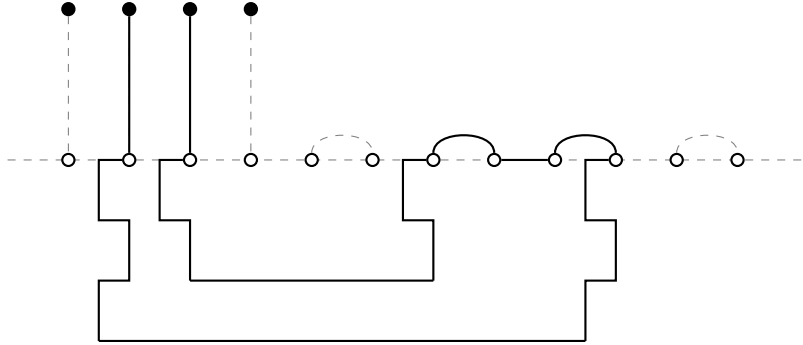

\end{proof}

\begin{lemma}\label{many-x-paths}
Let $G$ be a graph and let $A\subseteq V(G)$.
Let $W$ be a wall in $G-A$ with branch vertices $B$ such that all $B$-paths in $G-A$ are zero. Moreover, suppose that $\cT_W$ is a truncation of $\cT_{EP}$ and assume that the $\gamma$-paths from $A$ to any subset of $B$ have the \EPP\ with hitting set function $h$.
For any positive integer $k$, there are integers $f_{\ref{many-x-paths}}(k)$ and $g_{\ref{many-x-paths}}(k)$ such that if the size of $W$ is at least $f_{\ref{many-x-paths}}(k)$ and there are $g_{\ref{many-x-paths}}(k)$ $\gamma$-paths from $A$ to $B$ as well as that many $(-\gamma)$-paths from $A$ to $B$, then there are $k$ disjoint zero $A$-paths in $G$.
\end{lemma}

\begin{proof}
For a fixed positive integer $k$, let $h=h(100k^3)$, $p=g_{\ref{many-x-paths}}(k)=5(6h)^2+g$ and let $f_{\ref{many-x-paths}}(m,k)=4(5p)^2\cdot 3(h+5p+1)$.
Let $M$ be the set of vertices $v$ such that $v$ has three disjoint paths to $B\setminus \{v\}$. By Lemma \ref{zeroM}, for all $v\in M$, there is a $\gamma_v\in\Gamma_2$ such that all paths from $v$ to $B$ have weight $\gamma_v$. 

Let $P$ be any path from $A$ to $B$ with endvertices $a\in A$ and $b\in B$. Suppose that $P$ intersects the interior of four subdivided edges $e_1,\ldots,e_4$ of $W$ such that the first intersections of $P$ and these subdivided edges are in this order when starting in $a$. Each endvertex of $e_4$ is distinct from both endvertices of at least one of the subdivided edges $e_1,e_2$ and $e_3$. Let $w$ be an endvertex of $e_4$ and say it is distinct from the endvertices of $e_i$ for some $i\in [3]$.
Let $v_4$ be the closest intersection of $P$ and $e_4$ to $w$ and let $v_i$ be the last intersection of $P$ and $e_i$ before $v_4$, again both times on $P$ starting in $a$. The vertex $v_i$ has three disjoint paths to $B$; two of them in $e_i$ to the endvertices of $e_i$ and a third one through $P$ and $e_2$ to $w$. By Lemma \ref{zeroM}, any two paths from $v_i$ to $B$ have the same weight. We can reroute $P$ along $e_4$ to $w$, that is we obtain a path $P'=aPv_4e_4w$. Since the weight of $v_4e_4w$ is the same as the weight of $v_4Pb$, the weight of $P'$ and $P$ is the same. Slightly altering this argument, it is easy to see that also when $P$ intersects two subdivided edges $e_1$ and $e_2$, then we can reroute $P$ along $e_2$ to any of its endvertices that are disjoint from $e_1$.

Let $\cP_{\gamma}$ be the set of size $p$ of disjoint $\gamma$-paths from $A$ to $B$ such that $\sum_{P\in\cP_{\gamma}} E(P)\setminus E(W)$ is minimum. 
Let $e$ be a subdivided edge of $W$ that gets intersected by a path in $\cP_{\gamma}$ but $e$ is not one of the first three subdivided edge that any path in $\cP_{\gamma}$ intersects in its interior, starting in the endvertex in $A$. We claim that one of the endvertices of $e$ is an endvertex of a path in $\cP_{\gamma}$.

Let $P\in\cP_{\gamma}$ be the path that intersects $e$ closest to one of its endvertices $w$. Let $v$ be the closest intersection to $w$ of $P$ and $e$. We may assume that $v\neq w$, as we would be done then. This means that there is an edge of $P$ that comes after $v$ that does not lie in $W$. Now with the observation before, we can reroute $P$ along $e$ to $w$, to obtain a path $P'$ from $A$ to $B$ of the same weight as $P$. This path $P'$ contains fewer edges outside of $W$ and is still disjoint from all paths in $\cP_{\gamma}$ except $P$. Hence, replacing $P$ by $P'$ in $\cP_{\gamma}$ would have been a better choice for $\cP_{\gamma}$. This is a contradiction and proves the claim.

It follows that the interior of at most $4p$ many subdivided edges of $W$ get intersected by paths in $\cP_{\gamma}$ and $p$ branch vertices. Therefore, we find a subwall $W_1$ of $W$ of size $2\cdot 5p\cdot 3(h+5p+1)$ that is $(h+5p+1)$-contained and that is disjoint from all paths in $\cP_\gamma$. We claim:

\begin{equation}\label{manyxToWall}
\begin{minipage}[c]{0.8\textwidth}\em
There is a set $\cP_\gamma'$ of $100k^3$ disjoint $\gamma$-paths that nicely links $A$ to $W_1$.
\end{minipage}\ignorespacesafterend 
\end{equation}

Suppose otherwise. Let $N_1$ be the nails of $W_1$. By assumption the $\gamma$-paths have the \EPP\ and, hence, there is a set $X_\gamma\subseteq G-(W_1-N_1)$ of size at most $h$ that intersects all $\gamma$-paths from $A$ to $N_1$ in $G-(W_1-N_1)$. As the size of $\cP_\gamma$ is at least $p=5(6h)^2+h$, there are $5(6h)^2$ paths in $\cP_\gamma$ that are disjoint from $X_\gamma$; remove the other paths from $\cP_\gamma$. 
Since $W_0$ is $(h+5p+1)$-contained, there is a cycle $C$ in $W$ that encapsulates $W_0$, that is disjoint from $X_\gamma$ and all paths in $\cP_\gamma$ and that intersects all but $2(h+p+1)$ rows and as  many columns of $W$. 
Similarly, as the size of $W_1$ is larger than $h+5p+1$, there is a column of $W$ that intersects $N_1$ and that is disjoint from $X_\gamma$ and all paths in $\cP_\gamma$. In this column, there is a subpath $Q$ from $N_1$ to $C$ that is disjoint from $W_1-N_1$ because $C$ encapsulates $W_1$.
By pigeon principle and because at most two branch vertices lie in the same and column, at least $6h$ of the endvertices of paths in $\cP_\gamma'$ lie in different rows or different columns, say columns (the other case works analogously). The endvertices in $B$ of three paths lie in a column such that this column and the columns at distance $1$ are disjoint from $X_\gamma$. At least one endvertex $v$ of these three paths does not lie in the first or last column of $W$. Thus, there are two columns $T_1,T_2$ at distance $1$ from the column $T$ that contains $v$. Let $v$ be the endvertex of the path $P\in\cP_\gamma$.
Let $R_1,R_2,R_3$ be some rows that intersect $C$ and that are disjoint from $X_\gamma$ and $\cP_\gamma$. By choice of $C$, we can find such rows. There are three disjoint paths $P_1,P_2,P_3$ from $v$ to $(C\cup Q)\cap B$ that do not intersect $X-\gamma$ or $W_1$, in particular, these paths have weight zero. We use here that $C$ encapsulates $W_1$ and, hence, any column that intersects $W_1$ also intersects $C$ such that from any vertex in that column there is a path to $C$ that is disjoint from $W_1$. The idea here is to start in $v$ and go to $T,T_1$ and $T_2$ through three disjoint paths and follow each of these columns to either $C\cup Q$ or some $R_i$. In the case that we follow a column to some $R_i$, go through $R_i$ to $C\cup Q$. 

If $P$ is disjoint from at least one path $P_i$, construct a path from $A$ to $N_1$ by going through $P,P_i$ and then from the endvertex in $C\cup Q$ of $P_i$ through $C\cup Q$ to $N_1$. As $W$ is a zero wall, the weight of this path is the same as the weight of $P$, which is $\gamma$. This is a $\gamma$-path from $A$ to $N_1$ in $G-(W_1-N_1)$, which is a contradiction.
Now suppose that $P$ intersects each path $P_i$. We follow $P$ from its endvertex in $A$ until it intersects two of the paths $P_1,P_2,P_3$. We apply the reroute argument before (the one at the end of that paragraph) in the first subdivided edge of the second path $P_i$ that is being intersected. One of the two endvertices is closer to $v$ and the other to $C\cup Q$ on $P_i$. After rerouting we follow $P_i$ to $v$ or $C\cup Q$ depending on if we rerouted $P$ to the path closer to $v$ or $C\cup Q$ respectively. In the case that we followed $P_i$ to $v$, we go through the remaining path $P_j$ that is not being intersected by the rerouted path to obtain a path from $A$ to $C\cup Q$. With the same argument as before, we obtain a $\gamma$-path from $A$ to $N_1$ that is disjoint from $W_1-N_1$. Again this is a contradiction.
This proves (\ref{manyxToWall}).

Since $W_1$ is a subwall of $W$ of size $2\cdot 5p\cdot 3(h+5p+1)$, we may apply the argument of (\ref{manyxToWall}) to $W_1$ and $-\gamma$. Thus, there is a $(h+5p+1)$-contained subwall $W_2$ of $W_1$ of size at least $(h+5p+1)$ and a set $\cP_{-\gamma}'$ of $100k^3$ $(-\gamma)$-paths  that nicely links $A$ to $W_2$. Since $(h+5p+1)\geq 100k^3$, we can prolong the paths in $\cP_\gamma'$ to $W_2$ so that they still nicely link $A$ to $W_2$ and since $W$ is a zero wall, they are still $\gamma$-paths.

Let $a\in A$ be an endvertex of a path $P$ from $A$ to $B$ and let $m$ be the first vertex of $M$ that $P$ intersects, we call $aPm$ the starting segment of $P$.
Whenever two starting segments intersect, then one of the corresponding paths has to lie in $\cP_\gamma'$ and the other in $\cP_{-\gamma}'$.
Note that for $y\in\{\gamma,-\gamma\}$, the starting segment of a path in $\cP_y$ can intersect at most two paths in $\cP_{-y}$ since any vertex that is not in $M$ can be separated from $B$ by two vertices. For $y\in\{\gamma,-\gamma\}$, remove any path from $\cP_y$ that intersects more than $k$ starting segments of paths in $\cP_{-y}$.
Each set $\cP_y$ now contains at least $|\cP_y|-2\frac{|\cP_{-y}|}{k}\geq 10k^2$ paths. Pick $4k$ paths $P_1,\ldots,P_{4k}\in\cP_\gamma$ and remove the $4k^2$ paths from $\cP_{-\gamma}$ whose starting segment gets intersected by a path $P_i$ and also remove the $2k$ paths from $\cP_{-\gamma}$ that intersect the starting segments of $P_1,\ldots,P_{4k}$. Let $Q_1,\ldots,Q_{4k}$ be $4k$ of the remaining paths in $\cP_{-\gamma}$. Note that if $P_i$ and $Q_j$ intersect, they intersect after their starting segment. Now apply Lemma \ref{twotypespaths} to these two path systems to obtain a set of paths $\cP'\subseteq \{P_1,\ldots,P_{4k}\}$ of $2k$ paths and a set $\cQ'$ of $2k$ paths such that the paths in $\cP'$ and $\cQ'$ are paths that nicely link $A$ to $W_2$ and all their all paths in $\cP'\cup\cQ'$ are pairwise disjoint. Furthermore, with the note after Lemma \ref{twotypespaths}, we obtain that for any $Q\in\cQ'$, there is a path $Q_i$ such that $Q$ and $Q_i$ have the same starting segment. We already know that all paths in $\cP'$ are $\gamma$-paths. Let $Q\in\cQ'$ such that $Q$ and $Q_i$ have the same starting segment. Let $a\in A$ and $m\in M$ be the endvertices of this starting segment and let $b$ and $b_i$ be the endvertices of $Q$ and $Q_i$ in $B$ respectively. As $m\in M$, any two paths from $m$ to $B$ have the same weight and, thus, $mQ_ib_i$ and $mQb$ have the same weight. It follows that the weight of $Q$ and $Q_i$ is the same and, therefore, $Q$ is a $(-\gamma)$-path.

Now there is are subsets of $\cP^*\subseteq \cP'$ and $\cQ^*\subseteq \cQ'$ of size $k$ such that the endvertices of $\cP^*$ in the top row of $W_2$ come before the endvertices of $\cQ^*$, or the other way round. As $W_2$ is at least $5k$ and as we ordered the endvertices of $\cP^*$ and $\cQ^*$, simply connect the endvertices of the paths in $\cP^*$ to the endvertices of the paths in $\cQ^*$ disjointly through $W_2$. The results in $k$ disjoint zero $A$-paths and finishes the proof.
\end{proof}

\subsection{Proof}
With this we are now ready to prove the main theorem.

\mainThm*

\begin{proof}
 We already saw in the last chapter that these conditions are necessary,
so from now on we assume that $\Gamma$ is finite and for any $x,y\in\Gamma$ with $y\neq 0$, there is an $n\in\Z$ such that $2x+ny=0$.
Let $h(k)=3k+f_{\ref{AABpath}}(k)$, let $x(m,k)=g_{\ref{primeWall}}(g_{\ref{nonzeroLinkCase}}(m,k))$ and let $w(m,k)=\max\{f_{\ref{many-x-paths}}(k),2(x+5m\cdot h(g_{\ref{many-x-paths}}(k))\}$. 
We claim that the minimal function $f$ satisfying the following conditions, is a hitting set function for the zero $A$-paths.

\begin{enumerate}[(i)]
\item $f(k)\geq 3f_{\ref{tangleTOwall}}(f_{\ref{primeWall}}(m,2x(m,k)w(m,k),g_{\ref{nonzeroLinkCase}}(m,k)))+2f(k-1)+10$
\item $f(k)\geq 5m\cdot h(g_{\ref{many-x-paths}}(k))+x(m,k)$
\end{enumerate}

For a contradiction, suppose that this is not true. 
Let $(G,A,k)$ be a minimal counterexample for $f$ being a hitting set function.
By Lemma \ref{largetanglem}, there is a tangle $\cT_{EP}$ of order at least $f_{\ref{tangleTOwall}}(f_{\ref{primeWall}}(m,2x(m,k)w(m,k),g_{\ref{nonzeroLinkCase}}(m,k)))$ in $G-A$ such that for any separation $(C,D)$ of $G-A$ of smaller order, any zero $A$-path has to intersect $D-C$. Now we apply Lemma \ref{tangleTOwall} on $\cT_{EP}$ to find a wall $W$ in $G-A$ of size $f_{\ref{primeWall}}(m,2x(m,k)w(m,k),g_{\ref{nonzeroLinkCase}}(m,k))$ such that $\cT_W$ is a truncation of $\cT_{EP}$.
Next apply Theorem \ref{primeWall} to $G-A$ and $W$ to find a zero wall $W_0$ of size at least $2x(m,k)w(m,k)$ such that either there is a non-zero linkage of size $g_{\ref{nonzeroLinkCase}}(m,k)$ for $W_0$ or there is a vertex set $X$ of size $x(m,k)$ such that all paths in $G-X-A$ between the branch vertices of $W_0$ have weight zero. Additionally, the tangle $\cT_{W_0}$ is a truncation of $\cT_W$.
In the first case, with Lemma \ref{nonzeroLinkCase}, we can find $k$ disjoint zero $A$-paths in $G$. This is a contradiction to $(G,k,A)$ being a counterexample. So we assume that the latter case occured.
Let $B_0$ be the set of branch vertices of $W_0$. It holds that:

\begin{equation}\label{allZero}
\begin{minipage}[c]{0.8\textwidth}\em
There is a vertex set $X\subseteq V(G-A)$ such that in $G-X-A$ all $B_0$-paths have weight zero.
\end{minipage}\ignorespacesafterend 
\end{equation}

Let $G'=G-X$. 
Let $M$ be the set of vertices $v$ in $G'-A$ that have three disjoint paths to $B_0\setminus \{v\}$. 
As there are three disjoint paths from any branch vertex of a wall to other branch vertices of that wall, we may apply Lemma \ref{zeroM} to $B_0$ and $M$ in $G'$. It follows that:

\begin{equation}\label{Mzero}
\begin{minipage}[c]{0.8\textwidth}\em
For a vertex $v\in M$, there is a $\gamma_v\in\Gamma_2$ such that all paths from $v$ to $B_0$ in $G'-A$ have weight $\gamma_v$.
\end{minipage}\ignorespacesafterend 
\end{equation}

We call the vertices in $M$ anchors. Every vertex $v\in V(G'-A-M)$ can be separated from $M$ by at most two anchors. We call the minimum set of such separating anchors the \emph{separating anchors} of $v$. 
Note that by Lemma \ref{connectA}, any vertex that lies in a different component than $B_0\subseteq M$ cannot be part of any zero $A$-path, thus, can be removed from $G'$. 
We define a \emph{superedge} between two anchors $v$ and $w$ as the subgraph induced by $v,w$ and all vertices whose separating anchors are exactly $v$ and $w$, provided this subgraph is connected. 
In the case that there are vertices that are separated from $M$ by only one anchor, we allow $v$ and $w$ to be the same vertex. We call $v$ and $w$ the \emph{corresponding anchors} of that superedge.
We claim:

\begin{equation}\label{noFewVert}
\begin{minipage}[c]{0.8\textwidth}\em
There is no vertex set $Y$ of size at most $5m\cdot h(g_{\ref{many-x-paths}}(k))$ that separates a zero $A$-path $P$ from the branch vertices of any subwall $W_1$ of $W_0$ of size at least $2(x+5m\cdot h(g_{\ref{many-x-paths}}(k))$ in $G'$.
\end{minipage}\ignorespacesafterend 
\end{equation}

Suppose otherwise. Let $B_1$ be the set of branch vertices of $W_1$. Let $C'$ be the component of $G'-Y$ that contains $B_1\setminus Y$ and let $C=C' \cup Y$ and $D=G-C'-E(Y)$.
This induces a separation $(C,D)$ in $G$ of order at most $5m\cdot h(g_{\ref{many-x-paths}}(k))+x(m,k)$ such that $B_1$ lies in $C$ and $P$ in $D$. As $|X\cup Y|$ is smaller than the size of $\cT_{EP}$ and as $P$ does not intersect $C-D$, it follows that $(C,D)\in \cT_{EP}$.
On the other hand, since $|X\cup Y|$ is smaller than the size of $W_1$ and since $C$ contains all branch vertices of $W_1$, it holds that $(D,C)\in \cT_{W_1}$. This is a contradiction, because $\cT_{W_1}$ is a truncation of $\cT_{W_0}$, which is a truncation of $\cT_{EP}$. This proves (\ref{noFewVert}).

We immediately obtain the following:
\begin{equation}\label{noInSUper}
\begin{minipage}[c]{0.8\textwidth}\em
There is no zero $A$-path that is contained in the union of a superedge and $A$.
\end{minipage}\ignorespacesafterend 
\end{equation}

If this was not true we could separate a zero $A$-path from $B_0$ by removing just the corresponding anchors of that superedge, which are at most two. This contradicts (\ref{noFewVert}).

We want to show that:
\begin{equation}\label{ifxEPP}
\begin{minipage}[c]{0.8\textwidth}\em
If the $\gamma$-paths from $A$ to any subset $B^*$ of $B$ have the 
\EPP\ in $G'$ with hitting set function $h$, then there are either $k$ disjoint zero $A$-paths
in $G$ or a set of size at most $f(k)$ that intersects all those
paths.
\end{minipage}\ignorespacesafterend 
\end{equation}

Provided that this is true and we can show that these $\gamma$-paths have the \EPP, this would be a contradiction since $(G,k,A)$ was chosen to be a counterexample.
We will first show that this claim is indeed true,
and then we will prove that the assumption about the $\gamma$-paths is true as well.
This then finishes the proof.
So assume that the $\gamma$-paths from $A$ to any subset $B^*$ of $B$ have the 
\EPP\ in $G'$.

Let $W_1$ be any subwall of $W_0$ of size at least $2(x+5m\cdot h(g_{\ref{many-x-paths}}(k))$ with branch vertices $B_1$.
Suppose that for each $\gamma\in\Gamma$, there is a set $Y_\gamma$ of $h(g_{\ref{many-x-paths}})$ vertices that intersects either all $\gamma$-paths or all $(-\gamma)$-paths from $A$ to $B_1$ in $G'$. Let $Y=\cup_{\gamma\in\Gamma}Y_\gamma$; the size of $Y$ is at most $m\cdot h(g_{\ref{many-x-paths}}(k))$. 
Since $|X\cup Y|\leq f(k)$, the set $X\cup Y$ cannot intersect all zero $A$-paths in $G$. Let $P$ be a zero $A$-path that is disjoint from $X\cup Y$.
By (\ref{noFewVert}) and Menger's Theorem and because $|Y|+3\leq 5m\cdot h(g_{\ref{many-x-paths}}(k))$, there are three disjoint paths $Q_1,Q_2$ and $Q_3$ from $P$ to $B_1$ in $G'-Y$. Let the endvertices on $P$ of $Q_1,Q_2$ and $Q_3$ be $a_1,a_2$ and $a_3$ respectively and let the endvertices of $P$ itself be $v$ and $w$. We may assume that starting in $v$ and following $P$, we come across $a_1,a_2$ and $a_3$ in this order.
Note that for each $i\in\{1,2,3\}$, the path $Q_i$ is internally disjoint from $P$.
Now there are three disjoint paths from $a_2$ to $B_1\subseteq B_0$. These are $a_2Pa_1Q_1$, $Q_2$ and $a_2Pa_3Q_3$. It follows that $a_2\in M$ and by Lemma \ref{zeroM}, it holds that there is a $\gamma_{a_2}\in\Gamma_2$ such that all paths from $a_2$ to $B_1$ have weight $\gamma_{a_2}$. In particular, it holds that $P_2$ has weight $\gamma_{a_2}$. Let the weight of $vPa_2$ be $x$. As $P$ has weight zero, the path $a_2Pw$ has weight $-x$. 
We obtain the $(x+\gamma)$-path $vPa_2P_2$ from $A$ to $B_1$ and, since $\gamma=-\gamma$, also the $(-x-\gamma)$-path $wPa_2P_2$ from $A$ to $B_1$ in $G'-Y$. This is a contradiction because $Y_{x+\gamma_{a_2}} \subseteq Y$ intersects one of these two paths.

Let $W_1$ be a subwall of $W_0$ of size $w$ that is disjoint from $X$ with branch vertices $B_1$ (possible because $W_0$ large enough). We may assume that there is some $\gamma\in \Gamma$ such that there is no set $Y_\gamma$ of at most $h(g_{\ref{many-x-paths}})$ vertices that intersects all $\gamma$-paths or all $(-\gamma)$-paths. 
Since the $y$-paths have the \EPP\ with hitting set function $h$, it follows that there are $g_{\ref{many-x-paths}}$ disjoint $\gamma$-paths from $A$ to $B_1$ as well as that many $(-\gamma)$-paths from $A$ to $B_1$ in $G'$. 
Apply Lemma \ref{many-x-paths} here to find $k$ disjoint zero $A$-paths.
This finishes the proof of (\ref{ifxEPP}).

Now we only need to show that:
\begin{equation}\label{xEPP}
\begin{minipage}[c]{0.8\textwidth}\em
the $\gamma$-paths from $A$ to any subset $B^*\subseteq B$ in $G'$ have the \EPP\ with hitting set function $h(k)$.
\end{minipage}\ignorespacesafterend 
\end{equation}

For each anchor $v$ in $G'$, there is a $\gamma_v\in\Gamma_2$ such that all paths from $v$ to $B^*$ have weight $\gamma_v$.
We say there is a \emph{$\gamma$-path} from $a\in A$ to $v$ through a superedge $e^*$ if there is a $\gamma$-path from $a$ to $v$ in $G'[a\cup V(e^*)]$ that intersects the interior of $e^*$. In particular, it holds that $v\in V(e^*)$.
Otherwise if an anchor $v$ and a vertex $a\in A$ are adjacent and the edge between them has weight $\gamma$, we say that the edge between them is a \emph{direct $\gamma$-path}. 

Let $H=G'$ and in $H$, add an edge from $a\in A$ to $v\in M$ of weight $(\gamma-\gamma_v)$ if there is a $(\gamma-\gamma_v)$-path
from $a$ to $v$. Moreover, if this path intersects the interior of some superedge, remove the interior of this superedge from $H$. Then remove all edges incident to $A$ except for the ones we added in a previous step. Note that, technically, any direct $(\gamma-\gamma_v)$-path has been replaced by itself.

Assume there are $2k$ disjoint paths $P_1,\ldots, P_{2k}$ from $A$ to $B^*$ in $H$.
In each path $P_i$, replace again the edge that is incident to $A$ by the corresponding path that this edge replaced. Let $Q_1,\ldots, Q_{2k}$ be the resulting paths from $A$ to $B^*$ in $G'$.
We call the path that replaced an edge in $P_i$, the starting segment of $Q_i$.
Note that by construction of $H$, $Q_i$ really is a path and its weight is $\gamma$ because it can decomposed into a path of weight $\gamma-\gamma_v$ and a path of weight $\gamma_v$ for some $v\in M$.
Whenever two paths $Q_i$ and $Q_j$ intersect, the intersection lies in the starting segment of both paths. 
As the endvertices of the starting segment of $Q_i$ were also part of $P_i$, any intersection of two paths $Q_i$ and $Q_j$ has to lie in the interior of a superedge.
In particular, if the starting segment of $Q_i$ is a direct $(\gamma-\gamma_v)$-path, then $Q_i$ is disjoint from all other paths $Q_j$. 
It follows that any path $Q_i$ may only intersect exactly one other path $Q_j$.
Therefore, if we always remove one of the two paths $Q_i$ and $Q_j$ that intersect, we obtain a set of $k$ disjoint $\gamma$-paths from $A$ to $B^*$. Hence, we are done in this case.

Therefore, by Menger's theorem, we may assume that there is a vertex set $X_1\subseteq V(H)$ of size at most $2k$ that intersects all paths from $A$ to $B^*$ in $H$.
For $\gamma=\alpha+\gamma_v$, if there is an $\alpha$-path from $a\in A$ to $v$ through a superedge $e^*$, we say that $v$ is a \emph{$\gamma$-anchor for $a$ in $e^*$}. In this case we say that $e^*$ is a \emph{$\gamma$-superedge}. If there is a direct $\gamma$-path from $a\in A$ to $v$, we just say that $v$ is a \emph{$\gamma$-anchor for $a$}.
A \emph{closest $\gamma$-superedge} is a $\gamma$-superedge such that there is a path from one of its anchors to $B^*$ that does not intersect the interior of a $\gamma$-superedge.

The set $X_1$ is a subset of $V(G')$ and we remove $X_1$ from $G'$ and redefine the superedges in $G'-X_1$.
By using the property of $X_1$, any $\gamma$-path from $A$ to $B^*$ in $G'-X_1$ has to pass through the interior of a closest $\gamma$-superedge.
As $|X\cup X_1|<f(k)$, there still is at least one $\gamma$-path from $A$ to $B^*$ in $G'-X_1$. Let $C$ be the set of closest $\gamma$-superedges and let $C_A$ be the set of $\gamma$-anchors of superedges in $C$. Let $I$ be the union of the interiors of superedges in $C$.
We define the $B^*$-side in $G'-X_1$ as the components of $G'-I$ that contain vertices of $B^*$. The anchors in $C_A$ do not lie on the $B^*$-side, however, every closest $\gamma$-superedge has two anchors and the second does lie on the $B^*$-side.

We want to show that $C_A-C_A-B^*$-paths can be made into $\gamma$-paths from $A$ to $B^*$. 
Let $P$ be a $C_A-C_A-B^*$-path with an endvertex $c\in C_A$. We say that $P$ goes through a vertex $c_0\in C_A$ in the right direction of $e_0$ if $e_0$ is a closest $\gamma$-superedge with $\gamma$-anchor $c_0$ and if $P$ contains $c_0$ and if, starting in $c$, $P$ does not intersect the interior of $e_0$ after $c_0$.
If $P$ does go through a vertex $c_0$ in the right direction of some closest $\gamma$-superedge $e_0$, we may replace $cPc_0$ by the $(\gamma-\gamma_{c_0})$-path from $A$ to $c_0$ through $e_0$. This yields an $A$-$B^*$-path because $P$ went through $c_0$ in the right direction of $e_0$ and its length is $\gamma$ since any path from $c_0$ to $B^*$ has weight $\gamma_{c_0}$.
If $P$ does not pass through $c$ in the right direction of some closest $\gamma$-superedge $e$, then it has to intersect the interior of $e$ immediately after $c$.
As the anchor of $e$ that is distinct from $c$ lies on the $B^*$-side, the path $P$ intersects the $B^*$-side before picking up another vertex of $C_A$. Let $c_0$ be the first vertex in $C_A$ after $c$. As $P$ is on the $B^*$ side, it has to go through the interior of a closest $\gamma$-superedge $e_0$ to reach $c_0$. As each closest $\gamma$-superedge contains a vertex of $C_A$, the superedge $e_0$ has to contain $c_0$. Therefore, the path $P$ passes through $c_0$ in the right direction of $e_0$ (here we also use that $B^*\subseteq M$ and the interior of a superedge is disjoint from $M$). Assuming we find $k$ disjoint $C$-$C$-$B^*$-paths $P_1,\ldots,P_k$,
we make each path $P_i$ into a $\gamma$-path $Q_i$ from $A$ to $B^*$ with the argument before. As the paths $P_1,\ldots,P_k$ were disjoint also the paths $Q_1,\ldots, Q_k$ are disjoint. Thus, we are done if we find $k$ disjoint $C$-$C$-$M^*$-paths.

By Lemma \ref{AABpath}, we may assume that there is a vertex set $X_2$ of size at most $f_{\ref{AABpath}}(k)$ that intersects all $C_A-C_A-B^*$-paths. Let $C'$ be the set of closest $\gamma$-superedges in $G'-X_1-X_2$ and let $C_A'$ be the set of $\gamma$-anchors of superedges in $C'$. Note that $C_A'\subseteq C_A$. Let $I'$ be the union of the interiors of superedges in $C'$. Assume there are $c_1,c_2\in C_A'$ that are in the same component of $G'-X_1-X_2-I'$. So there has to be a path $P$ between $c_1$ and $c_2$ that does not intersect the $B^*$-side in $G'-X_1-X_2$. Then in $G'-X_1-X_2$, follow $P$ from $c_1$ to $c_2$, pass through the interior of a closest $\gamma$-superedge $e$ with $\gamma$-anchor $c_2$ to the other anchor of $e$.
From here we can go directly to $B^*$ without intersecting $P$ because we are on the $B^*$-side. This yields a $C$-$C$-$B^*$-path in $G'-X_1-X_2$, which is a contradiction.

Therefore, no two vertices in $C_A'$ are in the same component of $G'-X_1-X_2-I'$. Let $v\in C_A'$ be a $\gamma$-anchor of a closest $\gamma$-superedge $e\in C'$ and let $v$ lie in the component $O$ of $G'-X_1-X_2-I'$. Let $w$ be the other anchor of $v$.
Assume there is no $\gamma$-path from $A$ to $v$ in $O$. We claim that no $\gamma$-path from $A$ to $B^*$ intersects any vertices of $O$ or of the interior of $e$.
Suppose otherwise, let $P$ be such a path.
As there is no $\gamma$-path from $A$ to $w$ through $w$ and as $v$ does not lie on the $B^*$-side, the path $P$ has to contain a $\gamma$-path from $A$ to $v$ that lies completely in $O$. This is a contradiction and proves the claim. Therefore, we may remove $O$ and all vertices in the interior of $e$ in this case.

Let $C^*$ be the remaining closest $\gamma$-superedges and let $C_A^*$ be the set of $\gamma$-anchors of the superedges in $C^*$. Let $I^*$ be the interior of the superedges in $C^*$. Add an edge from $a\in A$ to $c\in C_A^*$ if there is a $\gamma$-path from $a$ to $c$ in $G'-X_1-X_2-I^*$. This edge represents a $\gamma$-path from $a$ to $c$. Remove all edges incident to $A$ that we did not add in the last step. Let $H'$ be this new graph.
If there are $k$ disjoint $A$-$B^*$-paths $P_1,\ldots, P_k$ in $H'$, then we replace the edge incident to $a\in A$ in $P_i$ by a $\gamma$-path in $G'-X_1-X_2$ that this edge represented to obtain a path $Q_i$. Now $Q_1,\ldots,Q_k$ are disjoint $\gamma$-paths by construction. By Menger's theorem, we may assume that there is a vertex set $X_3$ of size at most $k$ that intersects all $A$-$B^*$-paths in $H'$. 
It holds that $X_3\subseteq V(G'-X_1-X_2)$.
In $G'-X_1-X_2-X_3$ there is no $\gamma$-path from $A$ to $B^*$ anymore because any such path needs to intersect the interior of a closest $\gamma$-superedge and it needs to start on the side opposite the $B^*$-side. Such a path contains a path that $X_3$ intersects which would be a contradiction.
So this finishes the proof.
\end{proof}

\section{Conclusion}
With the main theorem, we obtain the following characterization of $A$-paths of weight $\gamma$ that have the EPP.

\mainCor*

\begin{proof}
As we have seen before these conditions are necessary.
So assume that they hold for $\Gamma$. 
First we will show that there is a $\delta\in\Gamma$ such that $2\delta=-\gamma$.

If $\gamma\in\Gamma_2$, let $y\neq \gamma$ be a non-zero element of $\Gamma$ (note that if $\gamma$ is the only non-zero element of $\Gamma$, then $A$-paths of length $\gamma$ are the non-zero $A$-paths and we can do that with Wollan's theorem). 
The order of $y$ has to be larger than $2$, as otherwise $2\gamma+ny\in\{0,y\}$ and never $=\gamma$.
It follows that $2y+n\gamma\in\{2y,2y+\gamma\}$. As the order of $y$ is larger than $2$ it cannot hold that $2y+\gamma=\gamma$. Hence, $2y=\gamma$ and we are done after setting $\delta=-y$.

Now let $m>2$ be the order of $\gamma$. If $m$ is odd, then $2(\frac{m+1}{2}\gamma)=\gamma+m\gamma=\gamma$ and we can choose $\delta=-\frac{m+1}{2}\gamma$. Otherwise if $m$ is even, it holds that for every $n\in\Z$ that $2\frac{m}{2}\gamma+n\frac{m}{2}\gamma \in \{0,\frac{m}{2}\gamma\}$. Since the order of $\gamma$ is  larger than $2$, it holds that $\frac{m}{2}\gamma\neq \gamma$. This contradicts the assumptions we made about $\Gamma$. This proves the claim.

We may remove all edges between two vertices of $A$ that does not have weight $\gamma$, since any such edge is never part of an $A$-path of length $\gamma$.
Any edge in $G[A]$ is an $A$-path of weight $\gamma$.
If there is a matching of size $k$ in $G[A]$, we found $k$ disjoint $A$-paths of weight $\gamma$. Otherwise we find a set $X_1$ of at most $2k$ vertices that intersect all edges in $G[A]$ and, thus, all edges between vertices of $A$ in $G$. Remove $X_1$ from $G$ and observe that there are no edges with both endvertices in $A$ anymore.

Now add $\delta$ to the weight of each edge that is incident to a vertex of $A$. Let $H$ be this graph. We want to show that any $A$-path of weight $\gamma$ in $G$ corresponds to a zero $A$-path in $H$.
If we can show that the zero $A$-paths have the \EPP\, it is quite easy to see that also the $A$-paths of length $\gamma$ have it.

Let $x,y\in\Gamma$ such that $y\neq 0$. We need to show that there is an $n\in\Z$ such that $2x+ny=0$. Let $m$ be the order of $\gamma$.
It holds that $m\gamma=0$. Moreover, there is an $n_1$ such that $2x+n_1y=\gamma$ and an $n_2$ such that $2\cdot 0+n_2y=\gamma$. Then if we choose $n=n_1+(m-1)n_2$, it follows that $2x+ny=2x+n_1\gamma+(m-1)n_2y=\gamma+(m-1)\gamma=0$. This implies that the zero $A$-paths with respect to $\Gamma$ have the \EPP.

Let $P$ be any $A$-path in $H$, then $P$ is also an $A$-path in $G$.
As there are no edges between two vertices of $A$, the path contains exactly two edges incident to $A$. By construction, the weight of $P$ in $G$ and the weight of $P$ in $H$ differs by exactly $2\delta=-\gamma$. So if the weight of $P$ is $\gamma$ in $G$, then its weight in $H$ is $\gamma+2\delta=0$. On the other hand a zero $A$-path in $H$ corresponds to an $A$-path of weight $\gamma$ in $G$.
So if there are $k$ disjoint zero $A$-paths in $H$, there are $k$ disjoint $A$-paths of weight $\gamma$ in $G$. If there is a set that intersects all zero $A$-paths in $H$, then this also intersects all $A$-paths of weight $\gamma$ in $G$. This finishes the proof.
\end{proof}

I want to note here that in contrast to Wollan's result on non-zero $A$-paths, the hitting set size for the $A$-paths of weight $\gamma\in\Gamma$ very much depends on $\Gamma$ (rather on the size of $\Gamma$). 

Now we can check whether the $A$-paths of length $d$ modulo $m$ have the \EPP.

\begin{corollary}
$A$-paths of length $d\ \m$ have the \EPP\ if and only if
for any $x,y\in\{0,\ldots,m-1\}$ with $y\neq 0$ there is an $n\in\Z$ such that $2x+ny\equiv d\ \m$.
\end{corollary}

\begin{proof}
Labelling all the edges in a graph with $1\in\Z_m$ implies that the weight of any path is equal to its length modulo $m$. Now we apply the characterization in Corollary \ref{mainCor} to finish the proof.
\end{proof}

Some calculations give us the known results:
The $A$-paths of length $0\ \textrm{mod}\ 2$, $1\ \textrm{mod}\ 2$, $0\ \textrm{mod}\ 4$ and $2\ \textrm{mod}\ 4$ have the \EPP, while the $A$-paths of length $1\ \textrm{mod}\ 4$, $3\ \textrm{mod}\ 4$ and $d\ \m$ for a non-prime $m\neq 4$ do not have the \EPP. As mentioned in the introduction, we also obtain the following result.

\begin{corollary}
Let $m$ be a fixed odd prime. The $A$-paths of length $d\ \m$ have the \EPP.
\end{corollary}

\begin{proof}
For primes $m$, all non-zero elements in $\Z_m$ are generators of $\Z_m$, that is for any $x,y\in\Z_m$ with $y\neq 0$ there is an $n\in\N$ such that $ny=x$, in particular,
there is an $n$ such that $ny=-2x+d$. This implies that $2x+ny=d$, which means we are done.
\end{proof}

Since multiple modulo constraints can always be written as a single modulo constraint (if there is a solution), this also characterizes all $A$-paths that have multiple modulo constraints.

All the proofs that I made, technically, also work for long $A$-paths, that is, for a fixed positive integer $\ell$, $A$-paths that have at least length $\ell$. 

\begin{theorem}
Let $\gamma\in\Gamma$. The long $A$-paths of weight $\gamma$ have the \EPP\ if and only if the $A$-paths of weight $\gamma$ do.
\end{theorem}

The counterexamples for $A$-paths of weight $\gamma$ already force the $A$-paths to be arbitrarily long (they have to traverse the whole grid from left to right). On the other hand all the proofs for the sufficiency can be easily adapted to the long case (although that is quite tedious). For example in Lemma \ref{nonzeroLinkCase}, we connect a set $A$ to a zero wall $W$ with a non-zero linkage and then through $W$ we construct zero $A$-paths. However, $W$ is a zero wall and can be made arbitrarily large (as long as it is bounded by some function in $k$). This means we are able to prolong the paths through $W$ by $\ell$ to get paths of length at least $\ell$ that still have weight $\gamma$. Of course, the size of the hitting set in this case will depend on $\ell$. 

Through Theorem \ref{ThomZero}, we immediately obtain that the zero cycles with respect to any finite Abelian group $\Gamma$ have the \EPP.

\begin{theorem}
Let $\Gamma$ be finite. The cycles of weight zero have the \EPP.
\end{theorem}

The proof works by finding a large tangle, very similar to the zero $A$-paths, and then applying Theorem \ref{tangleTOwall} to obtain a large wall. Then we use Theorem \ref{ThomZero} to find a large zero wall, which contains many disjoint zero cycles, which means we are done.

As a closing note, I wanted to note that the reason why I assumed $\Gamma$ to be Abelian is because otherwise there is no nice definition for an $A$-path of weight $\gamma$. For any non-Abelian group we can find elements $\gamma_1,\ldots,\gamma_n$ such that $\gamma_1+\ldots+\gamma_n=\gamma$ while $\gamma_n+\ldots+\gamma_1\neq \gamma$. 
If we label the edges of a path with $\gamma_1,\ldots,\gamma_n$, then in an undirected graph, there is no reason why we should choose one order of summation over the other.

\bibliographystyle{amsplain}
\bibliography{erdosposa}

\vfill

\small
\vskip2mm plus 1fill
\noindent
Version \today{}
\bigbreak

\noindent
Arthur Ulmer
{\tt <arthur.ulmer@uni-ulm.de>}\\
Institut f\"ur Optimierung und Operations Research\\
Universit\"at Ulm\\
Germany\\

\end{document}